\documentclass{amsart}

\usepackage{amsmath}
\usepackage{amsthm}
\usepackage{amssymb}
\usepackage{enumitem}\setlist[enumerate]{nosep,label=\textnormal{(\arabic*)}}
\usepackage[hidelinks]{hyperref}
\usepackage[capitalise]{cleveref}
\usepackage{mathrsfs}
\usepackage{mathtools}
\usepackage[new]{old-arrows}
\usepackage{seqsplit}
\usepackage{stmaryrd}
\usepackage{thmtools}
\usepackage[minimal]{yhmath}

\usepackage{tikz}
\usetikzlibrary{cd}
\usetikzlibrary{positioning, angles, quotes}
\declaretheorem[numberwithin=section,name=Theorem]{thm}
\declaretheorem[sibling=thm,name=Corollary]{cor}
\declaretheorem[sibling=thm,name=Proposition]{prop}
\declaretheorem[sibling=thm,name=Lemma]{lem}
\declaretheorem[sibling=thm,name=Conjecture]{conj}
\declaretheorem[sibling=thm,name=Claim]{claim}
\declaretheorem[sibling=thm,name=Definition,style=definition]{defn}
\declaretheorem[sibling=thm,name=Example,style=definition]{ex}
\declaretheorem[sibling=thm,name=Remark,style=remark]{rem}

\Crefname{thm}{Theorem}{Theorems}
\Crefname{lem}{Lemma}{Lemmas}
\Crefname{defn}{Definition}{Definitions}
\Crefname{claim}{Claim}{Claims}
\Crefname{ex}{Example}{Examples}
\Crefname{prop}{Proposition}{Propositions}

\newcommand{\SL}{\mathrm{SL}}

\newcommand{\pres}[2]{\langle{#1}\mid{#2}\rangle}

\newcommand{\FP}{\mathrm{FP}}

\renewcommand{\b}[1]{b^{(2)}_{#1}}

\newcommand{\Dk}[1]{\mathcal D_{k[{#1}]}}

\newcommand{\novk}[2]{\widehat{k[{#1}]}^{#2}}

\newcommand{\C}{\mathbb{C}}
\newcommand{\N}{\mathbb{N}}
\newcommand{\Q}{\mathbb{Q}}
\newcommand{\R}{\mathbb{R}}
\newcommand{\Z}{\mathbb{Z}}

\newcommand{\inv}{^{-1}}

\DeclareMathOperator{\cd}{cd}

\let\H\relax
\DeclareMathOperator{\H}{H}

\DeclareMathOperator{\link}{link}

\DeclareMathOperator{\supp}{supp}
\DeclareMathOperator{\Tor}{Tor}

\DeclareMathOperator{\Stab}{Stab}

\makeatletter
\newsavebox{\@brx}
\newcommand{\llangle}[1][]{\savebox{\@brx}{\(\m@th{#1\langle}\)}%
  \mathopen{\copy\@brx\mkern2mu\kern-0.9\wd\@brx\usebox{\@brx}}}
\newcommand{\rrangle}[1][]{\savebox{\@brx}{\(\m@th{#1\rangle}\)}%
  \mathclose{\copy\@brx\mkern2mu\kern-0.9\wd\@brx\usebox{\@brx}}}
\makeatother

\newcounter{comments}

\title{Coherent RFRS groups}

\author{Sam P.\  Fisher}
\address{Instituto de Ciencias Matem\'aticas, CSIC-UAM-UC3M-UCM}
\email{samuel.fisher@icmat.es}

\address{Instituto de Ciencias Matem\'aticas, CSIC-UAM-UC3M-UCM}
\email{marco.linton@icmat.es}

\author{Pablo S\'anchez-Peralta}
\address{Departamento de Matem\'aticas, Universidad Aut\'onoma de Madrid}
\email{pablo.sanchezperalta@uam.es}

\begin{document}

\maketitle
\centerline{\textit{With an appendix by Marco Linton}}

\begin{abstract}
    We prove that a finitely generated virtually RFRS group of cohomological dimension at most $2$ is coherent if and only if its second $L^2$-Betti number vanishes if and only if it is virtually free-by-cyclic. The non-vanishing of the second $L^2$-Betti number provides the first known global obstruction to coherence in any reasonably wide class of groups, allowing for proofs of incoherence without needing to exhibit explicit witnesses to incoherence. 
    
    As applications of this result, we completely characterise coherence among two-dimensional Coxeter groups, confirming conjectures of Jankiewicz and Wise, and show that incoherence is generic in groups of nonpositive deficiency, confirming a conjecture of Wise. We also find that, among virtually compact special groups of virtual cohomological dimension two, coherence is algorithmically decidable and is a quasi-isometry, measure equivalence, and profinite invariant.

    In an appendix, Marco Linton applies one of the main results to prove that cubulated locally quasi-convex hyperbolic groups are virtually free-by-cyclic, solving problems of Abdenbi--Wise and Wise in the cubulated case.
\end{abstract}

%%%
%%% INTRODUCTION
%%%
\section{Introduction}

A group is \emph{coherent} if all its finitely generated subgroups are finitely presented. On the first page of his survey on coherence \cite{Wise_anInvitation}, Wise writes that the ``main goal in this topic is to understand exactly which groups are coherent, and more realistically, to find useful characterizations of coherent groups''. The purpose of this article is to give such a characterisation in the class of virtually RFRS groups of cohomological dimension at most two. As suggested in  \cite[Section 3]{Wise_anInvitation}, coherence is a low-dimensional phenomenon, especially among hyperbolic groups, where all known coherent groups are of cohomological dimension two, with the notable exceptions coming from fundamental groups of closed hyperbolic $3$-manifolds. One can build coherent groups of arbitrarily high cohomological dimension by either taking virtual graphs of coherent groups with virtually polycyclic edge groups, or by a recent result of Jaikin-Zapirain, Linton, and the second author by taking one-relator products of locally indicable coherent groups \cite{JaikinLintonSanchez_OneRelProd} (though these constructions cannot as of yet yield high-dimensional hyperbolic coherent groups).

In dimension $2$, Wise conjectures that coherence is related to the nonpositive immersions property. A $2$-complex $X$ has \emph{nonpositive immersions} if for every compact, connected $2$-complex $Y$ and every cellular immersion $Y \looparrowright X$, either $\chi(Y) \leqslant 0$ or $\pi_1(Y) = \{1\}$. Wise conjectures that if $X$ has nonpositive immersions, then $\pi_1(X)$ is coherent, and speculates as to whether coherent groups of geometric dimension $2$ are always the fundamental group of a $2$-complex with nonpositive immersions (see \cite[Conjecture 1.10]{Wise_JussieuCoherenceNPI} and \cite[Section 21, Problems 9 and 11]{Wise_anInvitation}). A conjecture attributed to Gromov predicts that a group $G$ of geometric dimension at most $2$ is isomorphic to the fundamental group of a $2$-complex with nonpositive immersions if and only if $\b{2}(G) = 0$ \cite{WiseEnergy}, \cite[Section 16]{Wise_anInvitation}. We therefore have the following bold conjecture which gives interesting characterisations of coherence among $2$-dimensional groups. These equivalences are also either conjectured or raised as questions by Jaikin-Zapirain and Linton in \cite[Section 7]{JaikinLinton_coherence}.

\begin{conj}\label{conj: intro}
    Let $G$ be a group of geometric dimension at most $2$. The following are equivalent:
    \begin{enumerate}[label=\textnormal{(\arabic*)}]
        \item\label{item: conj coh} $G$ is coherent;
        \item\label{item: conj B2 = 0} $\b{2}(G) = 0$;
        \item\label{item: conj NPI} $G \cong \pi_1(X)$, where $X$ is a $2$-complex with nonpositive immersions.
    \end{enumerate}
\end{conj}

Since the nonpositive immersions property for $X$ implies that $\pi_1(X)$ is locally indicable \cite[Theorem 3.3]{Wise_JussieuCoherenceNPI}, \cref{conj: intro} predicts, in particular, that coherent two-dimensional groups are locally indicable. The purpose of this article is to establish the implication \ref{item: conj coh} $\Rightarrow$ \ref{item: conj B2 = 0} for residually finite rationally solvable (RFRS) groups, which form a rich subclass of locally indicable groups. In doing so, we will also virtually establish the implication \ref{item: conj coh} $\Rightarrow$ \ref{item: conj NPI}, which, combined with the work of Jaikin-Zapirain and Linton \cite{JaikinLinton_coherence}, essentially settles \cref{conj: intro} in the class of virtually RFRS groups.

The class of RFRS groups was introduced by Agol \cite{AgolCritVirtFib} in connection with Thurston's Virtual Fibring Conjecture for hyperbolic $3$-manifolds; it plays a central role in geometric group theory, and contains the class of special groups --- introduced by Haglund and Wise \cite{HaglundWise_special} --- as a proper subclass. Many interesting groups are known to be virtually special, including fundamental groups of hyperbolic $3$-manifolds \cite{AgolHaken,Wise_structureQCH}, limit groups \cite{Wise_structureQCH}, many one-relator groups \cite{Linton_oneRelHierarchy} including one-relator groups with torsion \cite{Wise_structureQCH}, balanced graphs of free groups with cyclic edge groups \cite{HsuWise_graphOfFree}, Coxeter groups \cite{HaglundWise_CoxeterSpecial}, small cancellation groups \cite{Wise_cubulatingSmallCancellation}, and hyperbolic free-by-cyclic groups \cite{HagenWise_freebyZ,Linton_FbyZembedding}, to name just a few.

\emph{Homological coherence} (over $\Q$) is an a priori weaker condition than coherence which says that all finitely generated subgroups of $G$ have the homological finiteness property $\FP_2(\Q)$. Our main result is the following.

\begin{thm}[{\cref{cor: hom coherence implies vanishing}}]\label{thm: A}
    If $G$ is a virtually RFRS group with $\cd_\Q(G) \leqslant 2$, then $G$ is homologically coherent over $\Q$ if and only if $\b{2}(G) = 0$.
\end{thm}

To the best of our knowledge, all known proofs that a group $G$ is incoherent go via exhibiting a specific finitely generated subgroup of $G$ that is not finitely presented. A new feature of \cref{thm: A} is that it provides a global obstruction to coherence (the second $L^2$-Betti number), and therefore gives a method for proving that a group $G$ is incoherent not relying on finding explicit witnesses to incoherence, which can be an arduous task. This will be exploited in \cref{thm: Coxeter}, where we determine which two-dimensional Coxeter groups are incoherent without finding explicit witnesses to their incoherence. 

In arbitrary cohomological dimension, we also find that higher $L^2$-Betti numbers obstruct group algebra coherence, as well as the property that all finitely generated subgroups are of type $\FP_\infty(\Q)$, which is a strong form of coherence. This latter property is possessed by all known examples of coherent groups, and in particular by all coherent special groups (see \cref{rem:v_special}). We refer the reader to \cref{thm: coherence and L2 Betti} for a stronger version of the following result. Recall that a ring is \emph{(left) coherent} if all of its finitely generated left ideals are finitely presented (for group algebras, left and right coherence are equivalent, so we will drop the left/right specification in this setting).

\begin{thm}[\cref{thm: coherence and L2 Betti}] \label{thm: B}
    Let $G$ be a virtually RFRS group. If $\b{n}(G) > 0$ for some $n > 1$, then $G$ has a finitely generated subgroup that is not of type $\FP_\infty(\Q)$. In particular, $\Q[G]$ is an incoherent ring.
\end{thm}

As a consequence of this result and a result of Linton from the appendix (see \cref{lem:type_F} and \cref{rem:v_special}), we obtain the following structure result for coherent special groups. We strongly believe the converse of the following also holds (see \cref{conj: LI coherence}).

\begin{cor}[{\cref{cor: coherent special}}] \label{cor: intro coherent special}
    Let $G$ be a finitely generated virtually special group. If $G$ is coherent, then $G$ is an iterated extension of a free group by cyclic or finite groups.
\end{cor}

In the appendix, Marco Linton uses \cref{thm: B} to prove the following result, which puts a strong restriction on the possible virtual cohomological dimensions of cubulated hyperbolic groups satisfying strong forms of coherence. This should be compared with Wise's prediction that coherence is a low-dimensional phenomenon among hyperbolic groups.

\begin{thm}[\cref{main,main_cor}]\label{thm: locally quasi convex}
    Let $G$ be a cubulated hyperbolic group.
    \begin{enumerate}
        \item If $G$ is locally hyperbolic, then $G$ is virtually an extension of a free product of free and surface groups by $\Z$; in particular, the virtual cohomological dimension of $G$ is at most $3$.
        \item If $G$ is locally quasi-convex, then $G$ is virtually free-by-cyclic; in particular, the virtual cohomological dimension of $G$ is at most $2$.
    \end{enumerate}
\end{thm}

There has been much recent progress in establishing the implication \ref{item: conj B2 = 0} $\Rightarrow$ \ref{item: conj coh} of \cref{conj: intro}. Jaikin-Zapirain and Linton proved that if $G$ is a locally indicable group of cohomological dimension $2$ and $\b{2}(G) = 0$, then $G$ is homologically coherent \cite[Theorem 1.2]{JaikinLinton_coherence}. In many cases, Jaikin-Zapirain--Linton are able to upgrade homological coherence to full coherence; in particular, they prove that one-relator groups are coherent, confirming a well known conjecture of Baumslag \cite[Theorem 1.1]{JaikinLinton_coherence}. More evidence for the \ref{item: conj B2 = 0} $\Rightarrow$ \ref{item: conj coh} direction of \cref{conj: intro} is provided by Kielak and Linton \cite[Theorem 1.1]{KielakLinton_FbyZ}, who proved that if $G$ is Gromov hyperbolic, virtually compact special, and $\cd_\Q(G) \leqslant 2$, then $G$ is virtually free-by-cyclic if and only if $\b{2}(G) = 0$. Free-by-cyclic groups are coherent by a result of Feighn--Handel \cite{FeighnHandel_FreeByZCoherent}, so the result of Kielak--Linton establishes a strong form of coherence for these groups. The first author obtained the same conclusion under the weaker assumption that $G$ be finitely generated, virtually RFRS, and of rational cohomological dimension at most $2$ \cite[Theorem A]{Fisher_freebyZ}.

Combined with these results, \cref{thm: A} yields the following characterisations of coherence among finitely generated virtually RFRS groups of rational cohomological dimension at most $2$. One equivalence we want to highlight is that such a group is coherent if and only if it is virtually free-by-cyclic (where the free kernel is not assumed to be finitely generated). Thus, our result provides a virtual converse to the Feighn--Handel theorem in the class of virtually RFRS groups.

\begin{cor}[\cref{cor: equivalences}]\label{cor: C}
    Let $G$ be a finitely generated virtually RFRS group with $\cd_\Q(G) \leqslant 2$. The following are equivalent:
    \begin{enumerate}
        \item\label{item: coherence intro} $G$ is coherent;
        \item\label{item: gp alg coh intro} $\Q[G]$ is coherent;
        \item\label{item: hom coherence intro} $G$ is homologically coherent over $\Q$;
        \item\label{item: H2 fg intro} $b_2(H;\Q)$ is finite for every finitely generated subgroup $H \leqslant G$;
        \item\label{item: second L2 fg all subgps intro} $\b{2}(H)$ is finite for every finitely generated subgroup $H \leqslant G$;
        \item\label{item: second L2 zero intro} $\b{2}(G) = 0$;
        \item\label{item: v free by Z intro} $G$ is virtually free-by-cyclic;
        \item\label{item: UG coherent} $\mathcal U(G)$ is a coherent $\Q[G]$-module;
        \item\label{item: weak dim 1 intro} $\mathcal U(G)$ is of weak dimension at most one as a $\Q[G]$-module;
        \item\label{item: NPI intro} $G$ has a subgroup of finite index isomorphic to $\pi_1(X)$ where $X$ is a $2$-complex with nonpositive immersions.
    \end{enumerate}
\end{cor} 

As mentioned above, these equivalences settle \cref{conj: intro} in the class of virtually RFRS groups, with the slight hiccup that we can only conclude that a finite-index subgroup of $G$ has a classifying space with nonpositive immersions in item \ref{item: NPI intro}. They also give positive resolutions to \cite[Question 2]{JaikinLinton_coherence} (with the same virtual caveat), and Questions 4 and 5 and Conjectures 3 and 6 from \cite{JaikinLinton_coherence}, in the class of virtually RFRS groups.

We comment on some of the definitions and context of the items appearing in \cref{cor: C}. A module is \emph{coherent} if all its finitely generated submodules are finitely presented; this generalises ring coherence, because a ring is left coherent if and only if it is coherent as a left module over itself. The question of whether group algebra coherence is equivalent to group coherence is an interesting open problem. Property \ref{item: H2 fg intro} has appeared explicitly in \cite[Corollary 1.6]{LouderWilton_Wcycles}, where it was shown that one-relator groups possess it (this is also a consequence of \cite{HelferWise_oneRelNPI}). The object $\mathcal U(G)$ appearing in items \ref{item: UG coherent} and \ref{item: weak dim 1 intro} is the algebra of operators affiliated to the group von Neumann algebra of $G$; its definition will be recalled in \cref{subsec: L2}. We also confirm the two-dimensional case of \cite[Conjecture 1.4]{Fisher_NovikovCohomology}, which predicts the equivalence of items \ref{item: coherence intro}, \ref{item: gp alg coh intro}, and \ref{item: weak dim 1 intro} for all RFRS groups of type $\FP$ (for arbitrary RFRS groups of type $\FP(\Q)$, we will show that coherence of $\Q[G]$ implies the weak dimension of $\mathcal U(G)$ is at most one, resolving one of the implications of \cite[Conjecture 1.4]{Fisher_NovikovCohomology} in general).

%%% APPLICATIONS
\subsection{Applications} \label{subsec: applic_intro}

As a corollary of \cref{thm: A}, we obtain a complete classification of coherent two-dimensional Coxeter groups, confirming Conjectures 3.4 and 4.7 and solving Problem 4.6 posed by Jankiewicz and Wise in \cite{JankiewiczWise_Coxeter}. Recall that a (finitely generated) \emph{Coxeter group} is a group with a presentation
\[
    \pres{v_1, \dots, v_n}{(v_iv_j)^{m_{ij}} : 1 \leqslant i \leqslant j \leqslant n},
\]
where $m_{ij} \in \{1, 2, 3, \dots, \infty\}$, $m_{ij} = 1$ if and only if $i = j$, and $(v_iv_j)^\infty$ means that there is no imposed relation between the generators $v_i$ and $v_j$. A \emph{Coxeter subgroup} is a subgroup generated by some subset of the generators $v_i$; these are Coxeter groups in their own right. The Coxeter group is \emph{two-dimensional} if and only if every subgroup of the form $\langle v_i, v_j, v_k \rangle$ is infinite for pairwise distinct integers $i$, $j$, $k$ (equivalently, that $\frac{1}{m_{ij}} + \frac{1}{m_{jk}} + \frac{1}{m_{ki}} \leqslant 1$ for all pairwise distinct $i$, $j$, $k$). The terminology is justified by the fact that a two-dimensional Coxeter group admits a finite-index subgroup with a nice two-dimensional classifying space. The \emph{Euler characteristic} of a finitely generated two-dimensional Coxeter group $G$ is 
\[
    \chi(G) \ = \ 1 - \frac{n}{2} + \sum_{1 \leqslant i < j \leqslant n} \frac{1}{2m_{ij}}.
\]
As mentioned above, Coxeter groups are virtually special, so they fall under the purview of our main result. The following result gives a very simple algorithm to check whether or not a two-dimensional Coxeter group is coherent, since there are only finitely many Coxeter subgroups.

\begin{thm}[\cref{thm: Coxeter}]\label{thm: F}
    Let $G$ be a two-dimensional Coxeter group. Then $G$ is coherent if and only if $\chi(H) \leqslant 0$ for every infinite Coxeter subgroup $H \leqslant G$.
\end{thm}

One direction of the characterisation is provided by \cite[Theorem 1.4]{JaikinLinton_coherence}. We use \cref{thm: A} to prove the converse, which states that a coherent two-dimensional Coxeter group must have $\b{2}(G) = 0$, and therefore that $\chi(H) \leqslant 0$ for all infinite finitely generated Coxeter subgroups. In \cref{thm: conj_Kasia_Dani}, we relate coherence to curvature properties of two-dimensional Coxeter groups, which is the subject of \cite[Problem 4.6 and Conjecture 4.7]{JankiewiczWise_Coxeter}.

In a different direction, we find that incoherence is generic (in the few-relator random model of \cite{ArzhantsevaOlshanskii_fewRelRandom}) among groups with nonpositive deficiency, confirming \cite[Conjecture 17.15(1)]{Wise_anInvitation}. The case $m = 2$ of the following result was obtained by Kielak, Kropholler, and Wilkes in \cite[Proposition 2.12]{KielakKrophollerWilkes_RandomL2}. We refer the reader to \cref{subsec: random} for a more precise statement.

\begin{cor}[\cref{thm: random coherence}]
    Fix integers $n \geqslant m > 1$ and let 
    \[
        G = \pres{s_1, \dots, s_m}{r_1, \dots, r_n}
    \]
    be a group whose relators are chosen uniformly at random among all words in $\{s_1^{\pm1},\dots,s_m^{\pm1}\}$ of length at most $l$. The probability that $G$ is incoherent tends to $1$ as $l$ tends to infinity.
\end{cor}

Using the work of L\"{o}h--Uschold on computability of $L^2$-Betti numbers \cite{LohUschold_L2computability}, we obtain the following algorithmic result.

\begin{cor}[\cref{cor: coh algorithm}]
    There is an algorithm that, given a virtually compact special group $G$ with $\cd_\Q(G) \leqslant 2$, determines whether or not $G$ is coherent.
\end{cor}

Additionally, we obtain some rigidity properties of coherence among virtually RFRS groups, which follow from the fact that the vanishing of $L^2$-Betti numbers has good invariance properties.

\begin{cor}[\cref{cor: rigidity}]
    Let $G$ and $H$ be finitely generated virtually RFRS groups of virtual cohomological dimension at most $2$. If $G$ and $H$ are
    \begin{enumerate}
        \item\label{item: QI_intro} quasi-isometric,
        \item\label{item: ME_intro} measure equivalent, or
        \item\label{item: profisom_intro} virtually compact special and profinitely isomorphic,
    \end{enumerate}
    then $G$ is coherent if and only if $H$ is coherent.
\end{cor}

Even though finite-presentability is well known to be a quasi-isometry invariant, it is an open problem whether the same holds for coherence \cite[Section 21, Problem 35]{Wise_anInvitation}. Part of the interest in item (2) lies in the fact that $\SL_2(\Z[\frac{1}{p}])$ is measure equivalent to the incoherent group $F_2 \times F_2$, and the (in)coherence of $\SL_2(\Z[\frac{1}{p}])$ is a famous open problem of Serre's \cite[pp.~734-735]{Proceedings_groupTheorySL2Zquestion}. We of course do not answer this question since $\SL_2(\Z[\frac{1}{p}])$ is not virtually RFRS. It seems that little is known about the profinite invariance of coherence; one result in this direction is a theorem of Morales, which states that coherence is a profinite invariant among finitely generated residually free groups \cite[Theorem D]{Morales_resFreeProfinite}, which form a subclass of the class of RFRS groups (though his result holds in arbitrary cohomological dimension). It has also been shown that the property of being virtually free-by-cyclic is a profinite invariant among virtually compact special groups of cohomological dimension at most two \cite[Corollary A.9]{HughesMollerVarghese_CoxterProfiniteProps}

To conclude, recall that a group $G$ is \emph{parafree} if it is residually nilpotent and there is some free group $F$ such that $G$ and $F$ have the same set of (isomorphism classes of) nilpotent quotients. Baumslag's \emph{Parafree Conjecture} predicts that any finitely generated parafree group $G$ satisfies $\H_2(G;\Z) = 0$. We show that in dimension $2$ the parafree conjecture is equivalent to the conjecture that parafree groups are coherent. Conjecturally, all parafree groups are of cohomological dimension at most $2$.

\begin{cor}[\cref{cor: parafree}]
    If $G$ is a finitely generated parafree group with $\cd(G) \leqslant 2$, then $G$ satisfies the Parafree Conjecture if and only if $G$ is coherent.
\end{cor}

%%% NOVIKOV HOMOLOGY
\subsection{Novikov homology of coherent groups}

In \cite{KielakRFRS}, Kielak relates the $L^2$-homology of a RFRS group $G$ with its homology with coefficients in Novikov rings. The Novikov ring $\widehat{\Q[G]}^\chi$ associated to a homomorphism $\chi \colon G \rightarrow \R$ is a completion of $\Q[G]$ along $\chi$, and can be thought of as the ring of infinite series in elements of $G$ with rational coefficients, with the property that for each $c \in \R$, only finitely many elements $g$ of the support satisfy $\chi(g) \leqslant c$ (see \cref{def: Novikov}).

The main step in our proof is showing that, among two-dimensional groups, coherence implies the vanishing of the second Novikov homology with respect to all integral characters. We then use Kielak's result to show that this implies the vanishing of the second $L^2$-Betti number (see the proof of \cref{cor: hom coherence implies vanishing}). The precise result we obtain is the following, which we believe to be of independent interest.

\begin{thm}[\cref{thm: nov vanish,prop: nov wd 1}]\label{thm: novikov}
    Let $G$ be a group and let $R$ be a ring.
    \begin{enumerate}
        \item If every finitely generated subgroup of $G$ is of type $\FP_\infty(R)$, then
        \[
            \H_i\left(G;\widehat{R[G]}^\chi\right) = 0
        \]
        for all $i > 1$ and all integral characters $\chi \colon G \rightarrow \Z$.
        \item If $R[G]$ is coherent, then
        \[
            \Tor_i^{R[G]}\left(M, \widehat{R[G]}^\chi \right) = 0
        \]
        for all $i > 0$, all integral characters $\chi \colon G \rightarrow \Z$, and all right $R[G]$-modules $M$. In other words, $\widehat{R[G]}^\chi$ is of weak dimension at most $1$ as an $R[G]$-module.
    \end{enumerate}
\end{thm}

We hope \cref{thm: novikov} might provide a new homological tool to prove incoherence outside the class of RFRS groups. It also has the following consequence (via Sikorav's theorem \cite{SikoravThese}) on the structure of the $\Sigma$-invariant of a coherent two-dimensional group, which appears to be new. We use $\Sigma_\Q^i(G)$ to denote the set of rational characters lying in the usual $i$th $\Sigma$-invariant $\Sigma^i(G)$ defined in \cite{BNSinv87,BieriRenzValutations}.

\begin{cor}[\cref{cor: BNS of coherent}] \label{cor: BNS_intro}
    Let $G$ be a finitely presented group of cohomological dimension at most $2$. If $G$ is coherent, then $\Sigma_\Q^1(G) = \Sigma_\Q^2(G)$.
\end{cor}

%%% ORGANISATION
\subsection*{Organisation of the paper}

In \cref{sec: prelims}, we recall the behaviour of homology with coefficients in an infinite product for modules of finite type and a result by Okun and Schreve, building on work of Kielak and Jaikin-Zapirain, concerning the homology of Novikov rings, central to our work. In \cref{sec: vanishing_homology}, we prove \cref{thm: novikov} showing that coherence implies vanishing of the Novikov homology. This is the key result that is used in the proof of \cref{thm: B} which produces the vanishing of $L^2$-Betti numbers in the presence of sufficiently strong finiteness conditions. From this, we obtain \cref{thm: A} providing the global obstruction to coherence. We also deduce \cref{cor: intro coherent special} and \cref{cor: BNS_intro} in this section. \cref{sec: application} contains the proof of \cref{cor: C} characterising coherent RFRS groups of cohomological dimension at most $2$, and the applications of the main results listed in \cref{subsec: applic_intro}. In particular, we classify coherent two-dimensional Coxeter groups. Finally, in Appendix \ref{sec: appendix}, Marco Linton proves \cref{thm: locally quasi convex}.

%%% ACKNOWLEDGMENTS
\subsection*{Acknowledgments}

The authors are very grateful to Andrei Jaikin-Zapirain and Dawid Kielak for useful conversations and comments on many previous versions of this article. The first author is supported by  the grant \seqsplit{CEX2023-001347-S} of the Ministry of Science, Innovation, and Universities of Spain. The second author is supported by the grants \seqsplit{PID2024-155800NB-C33} and \seqsplit{EUR2025-164928} of the Ministry of Science, Innovation, and Universities of Spain. The second author would also like to thank Dawid Kielak and the Mathematical Institute of the University of Oxford for the hospitality and support, since the foundations for this paper were laid while he was a guest at this institute.

%%%
%%% PRELIMINARIES
%%%
\section{Preliminaries} \label{sec: prelims}

Throughout, rings are assumed to be unital and associative, ring homomorphisms preserve the unit, and modules are left modules unless otherwise specified. If $R$ is a ring and $G$ is a group, then $R[G]$ denotes the corresponding group ring.

%%% FINITENESS PROPERTIES
\subsection{Finiteness properties of modules and groups}

\begin{defn}[Finiteness properties]
    Let $R$ be a ring. An $R$-module $M$ is of \emph{type $\FP_n$} for a non-negative integer $n$ if there is a projective resolution
    \[
        \cdots \to P_2 \to P_1 \to P_0 \to M \to 0
    \]
    such that $P_i$ is finitely generated for all $i \leqslant n$. If there exists a projective resolution as above such that all the modules $P_i$ are finitely generated, then $M$ is of \emph{type $\FP_\infty$}. If, additionally, there is some non-negative integer $n_0$ such that $P_i = 0$ for all $i \geq n_0$, then $M$ is of \emph{type $\FP$}. Analogous definitions exist for right $R$-modules.

    A group $G$ is of \emph{type $\FP_n(R)$} (resp.\ $\FP_\infty(R)$, resp. $\FP(R)$) if the trivial $R[G]$-module $R$ is of type $\FP_n$ (resp.\ $\FP_\infty$, resp. $\FP$).
\end{defn}

\begin{defn}[Dimensions]
    Let $R$ be a ring and $M$ a non-trivial $R$-module. The \emph{weak dimension} of $M$ is
    \[
        \sup\left\{ n \in \Z_{\geq 0} : \Tor_R^{n}(L,M) \neq \{0\} \ \text{for some right} \ R\text{-module} \ L \right\}.
    \]
    The \emph{projective dimension} of $M$ is the infimal length of a projective resolution for $M$. It is easy to see that the projective dimension is an upper bound for the weak dimension.

    Let $G$ be a group. The \emph{cohomological dimension} of $G$ over $R$, denoted $\cd_R(G)$, is the projective dimension of the trivial $R[G]$-module $R$.
\end{defn}

If $G$ is of type $\FP_n(R)$ and $\cd_R(G) \leqslant n$, then $G$ is of type $\FP(R)$ \cite[Proposition VIII.4.3]{BrownGroupCohomology}.

\begin{defn}[Augmentation ideal]
    Let $R$ be a ring and let $G$ be a group. The homomorphism 
    \[
        \varepsilon_R \colon R[G] \to R, \quad \sum_{g \in G} r_g g \mapsto \sum_{g \in G} r_g
    \]
    is called the \emph{augmentation map} of $R[G]$. The kernel of $\varepsilon_R$ is a two-sided ideal $I_{R[G]}$ called the \emph{augmentation ideal} of $R[G]$.
\end{defn}

It is not difficult to show that $G$ is of type $\FP_n(R)$ if and only if $I_{R[G]}$ is of type $\FP_{n-1}$ as an $R[G]$-module.

\begin{defn}[Coherence]
    A ring is \emph{(left) coherent} if all of its finitely generated left ideals are finitely presented (or, in the terminology above, if all of its left ideals of type $\FP_0$ are of type $\FP_1$). Similarly, a group is \emph{coherent} if all of its finitely generated subgroups are finitely presented. For a ring $R$, a group is \emph{homologically coherent} over $R$ if all of its finitely generated subgroups are of type $\FP_2(R)$.
\end{defn}

A group algebra $k[G]$ is left coherent if and only if it is right coherent. Hence, from now on we will just refer to $k[G]$ as \emph{coherent} if and only if it is left or right coherent.

We will use the following characterisation of homological coherence over $R$, which follows easily from the definitions above: $G$ is homologically coherent over $R$ if and only if $I_{R[H]}$ is a finitely presented $R[H]$-module for all finitely generated subgroups $H$ of $G$.

We record a standard fact about modules of type $\FP_\infty$ and include a proof for the reader's convenience.

\begin{lem}\label{lem: tor of products}
    Let $R$ be a ring and let $M$ be a right $R$-module. If $M$ is of type $\FP_{n+1}$ for some $n \geqslant 1$, then
    \[
        \Tor_i^R \left(M, \prod R \right) = 0
    \]
    for all $0 < i \leqslant n$, and the multiplication map
    \[
        M \otimes_R \prod R \to \prod M, \quad m \otimes (r_i)_{i \in I} \mapsto (m r_i)_{i \in I}
    \]
    is an isomorphism, where the direct product is taken over an arbitrary index set $I$.
\end{lem}
\begin{proof}
    Let 
    \[
        R^{d_{n+1}} \to R^{d_n} \to R^{d_{n-1}} \to \cdots \to R^{d_1} \to R^{d_0} \to M \to 0
    \]
    be a partial resolution of $M$ by finitely generated free $R$-modules. Since each $d_i$ is finite, the multiplication maps are isomorphisms
    \[
        R^{d_i} \otimes_R \prod R \xrightarrow{\cong} \prod R^{d_i}.
    \]
    Since the direct power functor $\prod \colon L \mapsto \prod L$ is exact, it follows that the chain complex
    \[
        \cdots \to R^{d_1} \otimes_R \prod R \to R^{d_0} \otimes_R \prod R \to 0
    \]
    is exact in degrees greater than zero. Hence, $\Tor_i^R(M, \prod R) = 0$ for all $0 < i \leqslant n$, proving the first claim.

    Since tensoring is right exact, the commutative diagram
    \[
        \begin{tikzcd}
            R^{d_1} \otimes_R \prod R \arrow[r]\arrow[d] & R^{d_0} \otimes_R \prod R \arrow[r]\arrow[d] & M \otimes \prod R \arrow[r]\arrow[d] & 0 \\
            \prod R^{d_1} \arrow[r] & \prod R^{d_0} \arrow[r] & \prod M \arrow[r] & 0
        \end{tikzcd}
    \]
    has exact rows. The middle and left vertical maps are isomorphisms since $d_0$ and  $d_1$ are finite. The second claim then follows from the Five Lemma. \qedhere
\end{proof}

We will use Chase's criterion for coherence of rings, which is closely related to the previous lemma, in the proof of \cref{prop: nov wd 1}.

\begin{thm}[{\cite{Chase_DirectProductsModules}}]\label{thm: Chase criterion}
    A ring $R$ is right coherent if and only if arbitrary products of flat left $R$-modules are flat.
\end{thm}

%%% L2 BETTI NUMBERS
\subsection{\texorpdfstring{$L^2$}{L²}-Betti numbers}\label{subsec: L2}

In this section we give a brief review of $L^2$-Betti numbers, emphasising what will be important for us. We refer the reader to \cite{Luck02} for more background on $L^2$-invariants. Let $G$ be a countable group and let $\ell^2(G)$ denote the Hilbert space of complex series in elements of $G$ with square-summable coefficients. The left and right multiplication actions of $G$ on itself extend to left and right actions of $G$ on $\ell^2(G)$. We can further extend the right action of $G$ on $\ell^2(G)$ to an action of $\C[G]$ on $\ell^2(G)$, and hence consider $\C[G]$ as a subalgebra of $\mathcal{B}(\ell^2(G))$, namely the {\it bounded linear operators on} $\ell^2(G)$. The {\it group von Neumann algebra} $\mathcal N(G)$ of $G$ is the algebra of bounded operators on $\ell^2(G)$ that commute with the left action of $G$. Note that $\mathcal N(G)$ contains the complex group algebra $\C[G]$. The ring $\mathcal{N}(G)$ is a finite von Neumann algebra. Moreover, $\mathcal{N}(G)$ satisfies the left (and right) Ore condition with respect to the set of non-zero-divisors (a result proved by S. K. Berberian in \cite{Berberian_vonNeumannOre}). We denote by $\mathcal{U}(G)$ the left (resp. right) classical ring of fractions named the {\it ring of unbounded operators affiliated to} $G$. Using the trace functional of $\mathcal{N}(G)$, one can define a dimension function on all $\mathcal{N}(G)$-modules (see \cite[Chapters 6.1 and 6.2]{Luck02}), and define for every non-negative integer $n$ the $n$th $L^2$-Betti number of $G$ to be
\[
    \b{n}(G) = \dim_{\mathcal N(G)} \H_n(G; \mathcal N(G)).
\]

The description of $L^2$-Betti numbers given so far relies on the involved definition of $\dim_{\mathcal N(G)}$. However, we will work with groups that satisfy the Strong Atiyah Conjecture, in which case the definition is far more straightforward. We give a formulation of the Strong Atiyah Conjecture for torsion-free groups. Given a subfield $k$ of $\C$, the \emph{Linnell ring over $k$}, denoted $\mathcal{D}(k[G])$, is the division closure of $k[G]$ in $\mathcal{U}(G)$.

\begin{conj}[The Strong Atiyah Conjecture over $k \subseteq \C$]
    Let $G$ be a torsion-free countable group, and let $k$ be a subfield of $\C$. The Linnell ring $\mathcal{D}(k[G])$ is a division ring.
\end{conj}

Linnell and Schick showed that this statement is equivalent to the usual formulation of the Strong Atiyah Conjecture in terms of the von Neumann dimension function (see \cite[Lemma 12.3]{LinnellZeroDivc} and \cite[Lemma 3]{SchickL2Int2002}). While the Strong Atiyah Conjecture is open in general, it has been established for many large classes of groups (including RFRS groups \cite[Theorem 1]{SchickL2Int2002}, see \cite[Corollary 1.2]{JaikinBaseChange}). For groups satisfying the Strong Atiyah Conjecture over $k \subseteq \C$, we have
\[
    \b{n}(G) = \dim_{\mathcal{D}(k[G])} \H_n(G; \mathcal{D}(k[G]))
\]
where $\dim_{\mathcal{D}(k[G])}$ is just the usual rank of a module over a division ring.

Our results will also hold with certain natural positive characteristic generalisation of the $L^2$-Betti numbers, which we introduce now. A group is \emph{locally indicable} if all its non-trivial finitely generated subgroups admit an epimorphism to $\Z$.

\begin{defn}[Hughes-free division ring]
    Let $G$ be a locally indicable group, let $k$ be a field, and let $k[G] \hookrightarrow \mathcal D$ be an embedding into a division ring. We will view $k[G]$ as a subring of $\mathcal D$. For $N \leqslant G$, let $\mathcal D_N$ denote the division closure of $k[N]$ in $\mathcal D$.
    
    We say that $\mathcal D$ is a \emph{Hughes-free} division ring for $k[G]$ if $\mathcal{D} = \mathcal{D}_G$ and the multiplication map
    \[
        \mathcal D_N \otimes_{k[N]} k[H] \to \mathcal D, \quad x \otimes y \mapsto xy
    \]
    is injective whenever $H \leqslant G$ is finitely generated and $N \triangleleft H$ is such that $H/N \cong \Z$.
\end{defn}

Hughes proved that if $G$ is locally indicable and $k[G]$ admits a Hughes-free division ring $\mathcal D$, then it is unique up to isomorphism of $k[G]$-rings \cite{HughesDivRings1970}. We thus speak of \emph{the} Hughes-free division ring of $k[G]$, and denote it by $\Dk{G}$. Jaikin-Zapirain and L\'opez-\'Alvarez showed that if $G$ is locally indicable, then $G$ satisfies the Strong Atiyah Conjecture over $\C$ \cite{JaikinLopezStrongAtiyah2020}. Moreover, the Linnell ring of $k[G]$ satisfies the Hughes-free condition for any subring $k \subseteq \C$, and thus, by uniqueness up to $k[G]$-isomorphism, $\mathcal{D}(k[G])$ is identified $\Dk{G}$ for a locally indicable group $G$. The following generalisation of $L^2$-Betti numbers for locally indicable groups naturally presents itself.

\begin{defn}[$\Dk{G}$-Betti numbers]
    Let $G$ be a locally indicable group and let $k$ be a field. If $\Dk{G}$ exists, then we define 
    \[
        b_n^{\Dk{G}}(G) = \dim_{\Dk{G}} \H_n(G; \Dk{G}).
    \]
    More generally, if $M$ is a right $k[G]$-module, then we define
    \[
        \beta_n^{\Dk{G}}(M) = \dim_{\Dk{G}} \Tor_n^{k[G]} (M, \Dk{G}).
    \]
    Hence, $b_n^{\Dk{G}}(G) = \beta_n^{\Dk{G}}(k)$.
\end{defn}

By \cite[Corollary 1.3]{JaikinZapirain2020THEUO}, $\Dk{G}$ exists for any field $k$ and any residually \{locally indicable and amenable\} group $G$. Note also that if $k$ is a field of characteristic zero, $b_n^{\Dk{G}}(G) = \b{n}(G)$.

In \cref{sec: application}, we will often use the connection between $L^2$-Betti numbers and Euler characteristic. Let $G$ be a group admitting a compact classifying space. We define the \emph{Euler characteristic} of $G$ by
\[
    \chi(G) = \sum_{n \geqslant 0} (-1)^n b_n(G;\Q),
\]
or equivalently as the alternating sum of the cells in each dimension (in any choice of compact classifying space for $G$). If $G$ contains a finite-index subgroup $H$ admitting a compact classifying space, then we put $\chi(G) := \chi(H)/[G:H]$. This is well-defined, a fact which is easier to see using the definition in terms of cells than the one in terms of Betti numbers. There is a very straightforward relationship between the $L^2$-Betti numbers of $G$ and its Euler characteristic.

\begin{lem}[{\cite[Theorem 1.35(2)]{Luck02}}]
    If $G$ is a group containing a finite-index subgroup $H$ admitting a compact classifying space, then
    \[
        \chi(G) = \sum_{n \geqslant 0} (-1)^n \b{n}(G).
    \]
\end{lem}

In particular, if $G$ (virtually) admits a compact two-dimensional classifying space, then $\chi(G) > 0$ implies $\b{2}(G) > 0$.

%%% KIELAK JAIKIN DESCRIPTION
\subsection{RFRS groups and Novikov rings} \label{subsec: RFRS}

\begin{defn}\label{def: RFRS}
    A group $G$ is called \emph{residually finite rationally solvable}, or \emph{RFRS}, if there is a chain of normal subgroups $G = G_0 \geqslant G_1 \geqslant G_2 \geqslant \dots$ of finite index such that $\bigcap_{i \geqslant 0} G_i = \{1\}$, and
    \[
        \ker(G_i \to \Q \otimes G_i/[G_i,G_i]) \leqslant G_{i+1}
    \]
    for all $i \geqslant 0$.
\end{defn}

It is easy to see that RFRS groups are residually \{locally indicable and virtually Abelian\} from the definition. Hence, it follows from \cite[Corollary 1.3]{JaikinZapirain2020THEUO} that $\Dk{G}$ exists for all fields $k$ and RFRS groups $G$.

Let $G$ be a group and let $R$ be a ring. If $x = \sum_{g \in G} a_g g$ is a formal series where $a_g \in R$ for all $g \in G$, then we define
\[
    \supp(x) = \{ g \in G : a_g \neq 0\}
\]
and call it the \emph{support} of $x$. Note that there is no requirement that $\supp(x)$ be finite here.

\begin{defn}\label{def: Novikov}
    Let $G$ be a group, let $R$ be a ring, and let $\chi \colon G \rightarrow \R$ be a group homomorphism. The \emph{Novikov ring} of $R[G]$ associated to $\chi$ is the ring
    \[
        \widehat{R[G]}^\chi := \left\{ x = \sum_{g \in G} a_g g \ : \ a_g \in R, \  \left|\supp(x) \cap \chi\inv((-\infty, c]) \right| < \infty \ \text{for all} \ c \in \R \right\}.
    \]
    The ring operations on $\widehat{R[G]}^\chi$ naturally extend those on $R[G]$.
\end{defn}

In \cite{KielakRFRS}, Kielak showed that the $L^2$-homology of a RFRS group is closely related to the homology of its finite-index subgroups with coefficients in Novikov rings. This principle was extended to $\Dk{G}$-Betti numbers over arbitrary fields $k$ by Jaikin-Zapirain \cite[Appendix]{JaikinZapirain2020THEUO}, and is neatly summarised by the following result, due to Okun and Schreve \cite{OkunSchreve_DawidSimplified}.

\begin{thm}[{\cite[Theorem 7.1]{OkunSchreve_DawidSimplified}}]\label{thm: L2 Novikov}
    Let $k$ be a field, let $G$ be a finitely generated RFRS group, and let $M$ be a right $k[G]$-module of type $\FP_{n+1}$. There is a finite index subgroup $H \leqslant G$ and a character $\chi \colon H \rightarrow \Z$ such that 
    \[
        \Tor_i^{k[H]}\left(M, \novk H \chi\right) = \left(\novk H \chi\right)^{[G:H] \beta_i^{\Dk{G}}(M)}
    \]
    for all $i \leqslant n$, where $M$ is viewed as a $k[H]$-module by restriction.
\end{thm}

\begin{rem}
    This is really just a rephrasing of \cite[Theorem 7.1]{OkunSchreve_DawidSimplified}, which is stated in terms of homology of a finite chain complex of finitely generated free $k[G]$-modules. To obtain this form, simply note that $M$ admits a resolution by finitely generated free $k[G]$-modules since it is of type $\FP_\infty$, and that $\Tor_n$ can be computed by truncating this resolution at length $n+1$ (so that it is the homology of a finite chain complex of finitely generated free modules).
\end{rem}

%%%
%%% VANISHING
%%%
\section{Vanishing of higher Novikov and \texorpdfstring{$L^2$}{L²}-homology} \label{sec: vanishing_homology}

We begin by proving that non-vanishing higher Novikov homology obstructs a strengthened form of homological coherence.

\begin{thm}\label{thm: nov vanish}
    Let $R$ be a ring and let $G$ be a group. If every finitely generated subgroup of $G$ is of type $\FP_{n+1}(R)$, then 
    \[
        \H_i \left(G;\widehat{R[G]}^\chi \right) = 0
    \]
    for all $1 < i \leqslant n$ and all integral characters $\chi \colon G \rightarrow \Z$.
\end{thm}
\begin{proof}
    The claim is trivial if $\chi$ is the zero character, so from now on we will assume that $\chi$ is surjective and that $t \in G$ is such that $\chi(t) = 1$. Let $N = \ker \chi$. The Lyndon--Hochschild--Serre spectral sequence with coefficients in $\widehat{R[G]}^\chi$ associated to the extension $1 \to N \to G \to \Z \to 1$ is
    \[
        E_{p,q}^2 = \H_p \left(\Z; \H_q \left(N; \widehat{R[G]}^\chi \right)\right) \quad \Rightarrow \quad \H_{p+q}\left(G;\widehat{R[G]}^\chi \right).
    \]
    We will prove that
    \begin{equation}\label{eq: kernel novikov homology}
        \H_q\left(N; \widehat{R[G]}^\chi \right) = 0
    \end{equation}
    for all $0 < q \leqslant n$. This will imply that 
    \[
        \H_p\left(\Z; \H_0\left(N; \widehat{R[G]}^\chi \right)\right) \cong \H_p\left(G; \widehat{R[G]}^\chi \right)
    \]
    for all $p \leqslant n$, and since $\cd_R(\Z) = 1$, this yields the desired result.
    
    We now prove \eqref{eq: kernel novikov homology}. Every element of $\widehat{R[G]}^\chi$ can be written as an infinite series $\sum_{i \in \Z} \lambda_i t^i$, where $\lambda_i \in R[N]$ and there is an $n_0 \in \Z$ such that  $\lambda_i = 0$ for all $i \leqslant n_0$. Hence, there is an inclusion
    \[
        \widehat{R[G]}^\chi \hookrightarrow \prod_{\Z} R[N], \quad \sum_{i \in \Z} \lambda_i t^i \mapsto (\lambda_i)_{i \in \Z}.
    \]
    From now we will write $\prod R[N]$ to denote $\prod_\Z R[N]$. There is another inclusion $\widehat{R[G]}^{-\chi} \hookrightarrow \prod_{\Z} R[N]$ defined similarly, and together, the inclusions induce the short exact sequence
    \[
        0 \to R[G] \to \widehat{R[G]}^\chi \oplus \widehat{R[G]}^{-\chi} \to \prod R[N] \to 0
    \]
    of $R[N]$-modules. Since $R[G]$ is a free $R[N]$-module, $\H_q(N;R[G]) = 0$ for all $q > 0$. Hence, the long exact sequence in the homology of $N$ associated to the above short exact sequence yields inclusions
    \[
        \H_q\left(N; \widehat{R[G]}^\chi \right) \oplus \H_q \left(N;\widehat{R[G]}^{-\chi} \right) \hookrightarrow \H_q \left(N; \prod R[N] \right),
    \]
    for all $q > 0$. The problem is thus reduced to proving that $\H_q(N; \prod R[N]) = 0$ for all $0 < q \leqslant n$.

    Consider the augmentation exact sequence of $R[N]$-modules
    \[
        0 \to I_{R[N]} \to R[N] \to R \to 0,
    \]
    which induces a long exact sequence in the functor $\Tor_q^{R[N]}(-,\prod R[N])$. Since $\Tor_q^{R[N]}(R[N],\prod R[N]) = 0$ for all $q > 0$, there are isomorphisms
    \begin{equation}\label{eq: q>1}
        \H_{q+1}\left(N;\prod R[N]\right) = \Tor_{q+1}^{R[N]}\left(R,\prod R[N]\right) \cong \Tor_q^{R[N]} \left(I_{R[N]}, \prod R[N]\right)
    \end{equation}
    for $q > 0$ and there is an exact sequence
    \begin{equation}\label{eq: q=1}
        0 \to \H_1\left(N; \prod R[N]\right) \to I_{R[N]} \otimes_{R[N]} \prod R[N] \to \prod R[N].
    \end{equation}
    
    Let $\{N_j\}$ be the directed set of all finitely generated subgroups of $N$. Then $N = \varinjlim_j N_j$, and therefore $I_{R[N]} = \varinjlim_j I_{R[N_j]}^N$, where $I_{R[N_j]}^N$ denotes the right ideal of $R[N]$ generated by $I_{R[N_j]}$; it is isomorphic to $I_{R[N_j]} \otimes_{R[N_j]} R[N]$ as a right $R[N]$-module. By assumption, $N_j$ is of type $\FP_{n+1}(R)$, so $I_{R[N_j]}$ is of type $\FP_n$ as an $R[N_j]$-module, and hence $I_{R[N_j]}^N$ is of type $\FP_n$ as an $R[N]$-module. Since Tor functors commute with direct limits, \cref{lem: tor of products} yields
    \[
        \Tor_q^{R[N]} \left(I_{R[N]}, \prod R[N]\right) \cong \varinjlim_j \Tor_q^{R[N]} \left(I_{R[N_j]}^N, \prod R[N]\right)  = 0
    \]
    for all $0 < q < n$. By \eqref{eq: q>1}, we have $\H_q(N;\prod R[N]) = 0$ for all $1 < q \leqslant n$.

    To conclude the proof, we must show that the map 
    \[
        I_{R[N]} \otimes_{R[N]} \prod R[N] \to \prod R[N]
    \]
    from \eqref{eq: q=1} is injective. Suppose that $x \in I_{R[N]} \otimes_{R[N]} \prod R[N]$ lies in the kernel. Since 
    \[
         I_{R[N]} \otimes_{R[N]} \prod R[N] \cong \varinjlim_j I_{R[N_j]}^N \otimes_{R[N]} \prod R[N],
    \]
    there is an index $j$ and an element $x' \in I_{R[N_j]}^N \otimes_{R[N]} \prod R[N]$ such that $x'$ maps to $x$ under the inclusion-induced homomorphism 
    \[
        I_{R[N_j]}^N \otimes_{R[N]} \prod R[N] \to I_{R[N]} \otimes_{R[N]} \prod R[N].
    \]
    The composition
    \[
        I_{R[N_j]}^N \otimes_{R[N]} \prod R[N] \to I_{R[N]} \otimes_{R[N]} \prod R[N] \to \prod R[N]
    \]
    is injective by \cref{lem: tor of products}. Thus, since $x'$ maps to zero under this homomorphism, we must have $x' = 0$. But then $x = 0$, as desired. \qedhere
\end{proof}

\cref{thm: nov vanish} has interesting consequences for the $\Sigma$-invariants of a coherent group. By Sikorav's theorem \cite{SikoravThese} and its higher dimensional versions (see e.g.\ \cite[Theorem 5.3]{Fisher_Improved}), the $i$th $\Sigma$-invariant of a group $G$ of type $\FP_n(R)$ is 
\[
    \Sigma^i(G;R) = \left\{ \chi \in \H^1(G;\R) \smallsetminus \{0\} \ : \ \H_j\left(G; \widehat{R[G]}^\chi\right) = 0 \ \text{for all} \ j \leqslant i \right\},
\]
whenever $i \leqslant n$. We put $\Sigma_\Q^i(G;R) = \Sigma^i(G;R) \cap \H^1(G;\Q)$. The $\Sigma$-invariants control finiteness properties of co-Abelian subgroups. The following is a representative result.

\begin{thm}[{\cite{BieriRenzValutations}}]
    Let $G$ be a group of type $\FP_n(R)$ for some ring $R$ and let $\chi \colon G \to \Z$ be a non-trivial character. Then $\ker \chi$ is of type $\FP_n(R)$ if and only if $\pm \chi \in \Sigma^n(G;R)$.
\end{thm}

Hence, it is clear that if $\chi \colon G \rightarrow \Z$ satisfies $\pm \chi \in \Sigma^1(G;R)$ but $\chi \notin \Sigma^2(G;R)$, then $G$ is incoherent, and the witness to incoherence is $\ker \chi$. In particular, if $G$ is a coherent group with a symmetric first $\Sigma$-invariant, then $\Sigma_\Q^1(G;R) = \Sigma_\Q^2(G;R)$. \cref{thm: nov vanish} shows that the symmetry requirement is unnecessary, at least in cohomological dimension two.

\begin{cor}\label{cor: BNS of coherent}
    Let $G$ be a group of type $\FP(R)$ and with $\cd_R(G) \leqslant 2$ for some ring $R$. If $G$ is homologically coherent over $R$, then $\Sigma_\Q^1(G;R) = \Sigma_\Q^2(G;R)$.
\end{cor}
\begin{proof}
    If a finitely generated subgroup of $G$ is of type $\FP_2(R)$, then it is automatically of type $\FP_\infty(R)$ since $\cd_R(G) \leqslant 2$. Hence, $\H_2(G; \widehat{R[G]}^\chi) = 0$ for all integral characters $\chi \colon G \rightarrow \Z$ by \cref{thm: nov vanish}. The result then immediately follows from the definition of $\Sigma_\Q^i(G;R)$. \qedhere
\end{proof}

If we strengthen the homological coherence assumption on $G$ in \cref{thm: nov vanish} to the assumption that $R[G]$ be coherent, we obtain an even stronger conclusion on the properties of $\widehat{R[G]}^\chi$ for integral characters.

\begin{prop}\label{prop: nov wd 1}
    Let $R$ be a ring and let $G$ be a group. If $R[G]$ is coherent, then $\widehat{R[G]}^\chi$ is of weak dimension at most $1$ as an $R[G]$-module for any integral character $\chi \colon G \rightarrow \Z$.
\end{prop}
\begin{proof}
    We must prove that for any right $R[G]$-module $M$, we have
    \[
        \Tor_i^{R[G]}\left(M,\widehat{R[G]}^\chi\right) = 0
    \]
    for all $i > 1$. Let $N = \ker \chi$. There is a spectral sequence
    \[
        E_{p,q}^2 = \H_p\left(\Z; \Tor_q^{R[N]}\left(M, \widehat{R[G]}^\chi\right)\right) \quad \Rightarrow \quad \Tor_{p+q}^{R[G]}\left(M, \widehat{R[G]}^\chi\right),
    \]
    so it will suffice to prove that $\Tor_q^{R[N]}(M, \widehat{R[G]}^\chi) = 0$ for all $q > 0$, since $\cd_R(\Z) = 1$.
    
    As in the proof of \cref{thm: nov vanish}, we have a short exact sequence
    \[
        0 \to R[G] \to \widehat{R[G]}^\chi \oplus \widehat{R[G]}^{-\chi} \to \prod R[N] \to 0
    \]
    where the product is indexed over $\Z$. Since $R[N]$ is coherent, $\prod R[N]$ is flat as an $R[N]$-module by Chase's criterion (\cref{thm: Chase criterion}). But so is $R[G]$, being a free $R[N]$-module, so the long exact sequence in $\Tor^{R[N]}(M,-)$ associated to the short exact sequence above immediately yields
    \[
        \Tor_q^{R[N]}\left(M, \widehat{R[G]}^\chi\right) = 0
    \]
    for all $q > 0$, as desired. \qedhere
\end{proof}

Leveraging the relationship between Novikov homology and $L^2$-homology (and its positive characteristic analogues), we obtain the following result.

\begin{thm}\label{thm: coherence and L2 Betti}
    Let $k$ be a field and let $G$ be a RFRS group.
    \begin{enumerate}
        \item\label{item: all subgps FPinf} If every finitely generated subgroup of $G$ is of type $\FP_{n+1}(k)$, then
        \[
            b_i^{\mathcal{D}_{k[G]}}(G) = 0
        \]
        for all $1 < i \leqslant n$.
        \item\label{item: gp alg coh weak dim one} If $k[G]$ is coherent, then $\Dk{G}$ is of weak dimension at most $1$ as a $k[G]$-module.
        \item\label{item: finite cd subnormal} If $G$ is of type $\FP(k)$ and every finitely generated subgroup of $G$ is of type $\FP(k)$, then there is a subnormal series
        \[
            G_n \trianglelefteqslant G_{n-1} \trianglelefteqslant \cdots \trianglelefteqslant G_1 \trianglelefteqslant G_0 = G
        \]
        such that $G_n$ is free and every consecutive quotient $G_i/G_{i+1}$ is cyclic or finite. In particular, $\Dk{G}$ is of weak dimension at most $1$ as a $k[G]$-module.
    \end{enumerate}
\end{thm}
\begin{proof}
    \ref{item: all subgps FPinf} Let $\{G_\lambda\}_{\lambda \in \Lambda}$ be the directed set of all finitely generated subgroups of $G$, so that $G = \varinjlim_\lambda G_\lambda$. We begin by proving that
    \[
        \H_i\left(G_\lambda; \Dk{G_\lambda}\right) = 0
    \]
    for all $1 < i \leqslant n$ and all $\lambda \in \Lambda$. Since $G_\lambda$ is of type $\FP_{n+1}(k)$, we can apply \cref{thm: L2 Novikov} to obtain a finite-index subgroup $H_\lambda \leqslant G_\lambda$ and a non-trivial integral character $\chi \colon H_\lambda \to \Z$ such that
    \[
        \H_i\left( H_\lambda ; \novk{H_\lambda}{\chi} \right) = \left(\novk{H_\lambda}{\chi}\right)^{[G_\lambda : H_\lambda]  b_i^{\Dk{G_\lambda}}(G_\lambda)}
    \]
    for all $i \leqslant n$. But we know, from \cref{thm: nov vanish}, that $\H_i( H_\lambda ; \novk{H_\lambda}{\chi}) = 0$ for all $1 < i \leqslant n$, so we must have $b_i^{\Dk{G_\lambda}}(G_\lambda) = 0$ for all $1 < i \leqslant n$, proving the claim.

    Now, $I_{k[G]} = \varinjlim I_{k[G_\lambda]}^G$, so for all $0 < i \leqslant n - 1$ we have
    \begin{align*}
        \H_{i+1}(G;\Dk{G}) &= \Tor_i^{k[G]}(I_{k[G]},\Dk{G}) \\
        &= \varinjlim \Tor_i^{k[G]}(I_{k[G_\lambda]}^G,\Dk{G}) \\
        &= \varinjlim \Tor_i^{k[G_\lambda]}(I_{k[G_\lambda]},\Dk{G}) \\
        &= \varinjlim \Tor_i^{k[G_\lambda]}(I_{k[G_\lambda]},\Dk{G_\lambda}) \otimes_{\Dk{G_\lambda}} \Dk{G} \\
        &= \varinjlim \H_{i+1}(G_\lambda;\Dk{G_\lambda}) \otimes_{\Dk{G_\lambda}} \Dk{G} \\
        &= 0
    \end{align*}
    using Shapiro's lemma, and the facts that Tor commutes with direct limits and $\Dk{G}$ is flat over $\Dk{G_\lambda}$. This proves that $b_i^{\Dk{G}}(G) = 0$ for all $1 < i \leqslant n$.

    \smallskip

    \ref{item: gp alg coh weak dim one} Let $M$ be an arbitrary $k[G]$-module; our goal is to show that 
    \[
        \Tor_i^{k[G]}\left(M,\Dk{G}\right) = 0
    \]
    for all $i > 1$. The following claim reduces the problem to the case where $G$ is finitely generated, so that we may apply \cref{thm: L2 Novikov}).
    
    \begin{claim}
        If $\Tor_i^{k[G_0]}(M,\Dk{G_0}) = 0$ for every finitely generated subgroup $G_0$ of $G$, then $\Tor_i^{k[G]}(M,\Dk{G}) = 0$.
    \end{claim}
    \begin{proof}
        Gr\"ater shows that $\Dk{G}$ satisfies the strong Hughes-free property \cite[Corollary 8.3]{Grater20}, which means that the multiplication map $k[G] \otimes_{k[G_0]} \Dk{G_0} \to \Dk{G}$ is injective for all subgroups $G_0 \leqslant G$. Hence, $\Dk{G} = \varinjlim k[G] \otimes_{k[G_0]}\Dk{G_0}$ as a $k[G]$-module where the directed set ranges all finitely generated subgroups $G_0$ of $G$. By Shapiro's Lemma, this implies that
        \[
            \Tor_i^{k[G]}(M, \Dk{G}) \cong \varinjlim \Tor_i^{k[G_0]}(M, \Dk{G_0}) = \{0\},
        \]
        as desired. \renewcommand\qedsymbol{$\diamond$}\qedhere
    \end{proof}
    
    From now on, we assume that $G$ is finitely generated. Write $M = \varinjlim_\lambda M_\lambda$, where $\{M_\lambda\}_{\lambda \in \Lambda}$ is a directed system of finitely presented $k[G]$-modules. We will show  that 
    \[
        \Tor_i^{k[G]}(M_\lambda, \Dk{G}) = 0
    \]
    for all $i > 1$ and $\lambda \in \Lambda$. Since $k[G]$ is coherent, $M_\lambda$ is of type $\FP_\infty$, so we can apply \cref{thm: L2 Novikov} to find a finite-index subgroup $H \leqslant G$ such that 
    \[
        \Tor_i^{k[H]}\left(M_\lambda, \novk{H}{\chi}\right) = \left(\novk{H}{\chi}\right)^{[G : H] \beta_i^{\Dk{G}}(M_\lambda)}.
    \]
    But $\novk H \chi$ is of weak dimension at most $1$ as a $k[H]$-module by \cref{prop: nov wd 1}, so the modules in the above line must be zero for $i > 1$, and therefore we get $\Tor_i(M_\lambda, \Dk{G}) = 0$ for all $i > 0$, as claimed.

    Finally,
    \[
        \Tor_i^{k[G]}(M, \Dk{G}) \cong \varinjlim_\lambda \Tor_i^{k[G]}(M_\lambda, \Dk{G}) = 0,
    \]
    so $\Dk{G}$ is of weak dimension at most $1$ as a $k[G]$-module.

    \smallskip

    \ref{item: finite cd subnormal} Since the property ``all finitely generated subgroups are of type $\FP_\infty(k)$" passes to subgroups, item \ref{item: all subgps FPinf} implies that $b_i^{\Dk{H}}(H)$ vanishes for all subgroups $H \leqslant G$ and all $i > 1$. The desired conclusions then follow from \cite[Theorem 5.7]{Fisher_NovikovCohomology}. \qedhere
\end{proof}

In dimension two, we obtain our main theorem, which gives a complete characterisation of coherence in the class of RFRS groups. To obtain \cref{thm: A} from the introduction as a corollary of this result, take $k = \Q$ and note that coherence and vanishing of $L^2$-Betti numbers are commensurability invariants. 

\begin{cor}\label{cor: hom coherence implies vanishing}
    Let $G$ be a RFRS group and let $k$ be a field with $\cd_k(G) \leqslant 2$. Then $G$ is homologically coherent if and only if $b_2^{\Dk{G}}(G) = 0$.
\end{cor}
\begin{proof}
    If $G$ is homologically coherent, then every finitely generated subgroup of $G$ is of type $\FP(k)$, since $\cd_k(G) \leqslant 2$. Hence, item \ref{item: all subgps FPinf} of \cref{thm: coherence and L2 Betti} implies that $\H_2(G;\Dk{G}) = 0$. Conversely, if $\H_2(G;\Dk{G}) = 0$, then $G$ is virtually free-by-cyclic by \cite[Corollary 3.9]{Fisher_freebyZ}, and hence (homologically) coherent by \cite{FeighnHandel_FreeByZCoherent}. \qedhere
\end{proof}

As shown by Linton in the appendix (see \cref{lem:type_F} and \cref{rem:v_special}), if $G$ is virtually compact special and homologically coherent over $\Q$, then all finitely generated subgroups of $G$ are of type $\FP(\Q)$. The compact special hypothesis is not needed, since any finitely generated special group embeds into a finitely generated right-angled Artin group. A right-angled Artin group has a hierarchy inducing the required hierarchy on its finitely generated subgroups, so \cref{rem:v_special} still applies. We thus obtain the following structural result for coherent special groups.

\begin{cor}\label{cor: coherent special}
    Let $G$ be a finitely generated virtually special group. If $G$ is homologically coherent over $\Q$, then there is a subnormal series 
    \[
        G_n \trianglelefteqslant G_{n-1} \trianglelefteqslant \cdots \trianglelefteqslant G_1 \trianglelefteqslant G_0 = G
    \]
    where $G_n$ is free and every consecutive quotient $G_i/G_{i+1}$ is either cyclic or finite.
\end{cor}
\begin{proof}
    This follows from the discussion in the paragraph preceding the corollary and an application of \cref{thm: coherence and L2 Betti}\ref{item: finite cd subnormal}. \qedhere
\end{proof}

\cref{cor: coherent special} motivates the following conjecture, which proposes a classification of coherence in the class of locally indicable groups (compare \cite[Conjecture 1.4]{Fisher_NovikovCohomology}).

\begin{conj}\label{conj: LI coherence}
    Let $G$ be a virtually locally indicable group. The following are equivalent:
    \begin{enumerate}
        \item\label{item: LI coh} $G$ is coherent;
        \item\label{item: LI FPinf coh} every finitely generated subgroup of $G$ is of type $\FP_\infty(\Q)$;
        \item\label{item: LI all L2 vanish} every subgroup $H \leqslant G$ satisfies $\b{n}(H) = 0$ for all $n > 1$.
    \end{enumerate}
\end{conj}

Thus, for virtually RFRS groups, the implication \ref{item: LI FPinf coh} $\Rightarrow$ \ref{item: LI all L2 vanish} holds by \cref{thm: coherence and L2 Betti}, and additionally for special groups the implication \ref{item: LI coh} $\Rightarrow$ \ref{item: LI FPinf coh} holds by \cref{lem:type_F} and \cref{rem:v_special} (as well as the remarks preceding \cref{cor: coherent special}). The conjecture for virtually RFRS groups of cohomological dimension at most two is completely resolved in \cref{cor: equivalences}.

We conclude the section with two examples. First we mention a potential limitation of this method of proving incoherence, even in the class of locally indicable groups of cohomological dimension two. A group $G$ is \emph{free-by-free} if it fits into an extension $1 \to F_m \to G \to F_n \to 1$ for some integers $m,n \geqslant 0$, where $F_n$ is the free group on $n$ generators. A free-by-free group has \emph{excessive homology} if $b_1(G) > n$. In \cite[Theorem 4.1]{KrophollerVidussiWalsh_SurfaceBySurface}, Kropholler, Vidussi, and Walsh give an example of a free-by-free group such that all its subgroups of finite index have no excessive homology. While free-by-free groups are expected to be incoherent \cite[Problem 24]{Wise_anInvitation}, this example shows that Novikov homology of finite-index subgroups cannot detect this.

\begin{ex}
    Let $G = F_m \rtimes F_n$ be a free-by-free group with no excessive homology and let $k$ be a field. We claim that 
    \[
        \H_2\left(G; \novk G \chi\right) = 0
    \]
    for all integral characters $\chi \colon G \to \Z$.

    Fix an integral character $\chi \colon G \rightarrow \Z$. The Lyndon--Hochschild--Serre spectral sequence associated to the extension $1 \to F_m \to G \to F_n \to 1$ has the form
    \[
        E_{p,q}^2 = \H_p\left(F_n; \H_q\left(F_m; \novk G \chi\right)\right) \quad \Rightarrow \quad \H_{p+q}\left(G; \novk G \chi\right).
    \]
    We show that $\H_q(F_m; \novk G \chi) = 0$ for all $q > 0$, which will imply the result, since $\cd_k(F_n) = 1$. Note we only need to check this for $q = 1$, since $\cd_k(F_m) = 1$.

    Since $G$ has no excessive homology, $\chi$ must factor through $G \rightarrow F_n$, and therefore $N = \ker \chi$ contains $F_m$. As in the proof of \cref{thm: nov vanish}, it is enough to prove that 
    \[
        \H_1\left(F_m; \prod k[N]\right) = 0.
    \]
    Since $k[N]$ is flat as a $k[F_m]$-module and $k[F_m]$ is coherent (in fact it has the much stronger property of being a free ideal ring \cite{cohn06FIR}), it follows that $\prod k[N]$ is flat as a $k[F_m]$-module (\cref{thm: Chase criterion}), and we obtain the claimed vanishing.
\end{ex}

Next, we show that while vanishing higher $L^2$-homology characterises coherence in the class of two-dimensional RFRS groups, there are incoherent solvable RFRS groups of infinite cohomological dimension. Such groups have $\mathcal D_{\Q[G]}$ of weak dimension zero as a $\Q[G]$-module, and therefore all the $L^2$-Betti numbers of all their subgroups vanish.

\begin{ex}\label{ex: Z wreath Z}
    Let $G = \Z \wr \Z$. First, since $G$ is locally indicable (and therefore has unique products), $k[G]$ is a domain for any field $k$, and because $G$ is solvable, $k[G]$ is in fact an Ore domain by \cite{Tamari_Ore}. Hence, $\Dk{G}$ exists and is of weak dimension zero as a $k[G]$-module. But $G$ is incoherent, as it is itself finitely generated but not finitely presented (see, e.g., \cite{Baumslag_wreathProdFP}).

    We now show that $G$ is RFRS, using the fact that a finitely generated group is RFRS if and only if it is residually \{poly-$\Z$ and virtually Abelian\} \cite[Theorem 6.3]{OkunSchreve_DawidSimplified}. Let $(v,t) \in \left(\bigoplus_{i \in \Z} \Z\right) \rtimes \Z = G$ be a non-trivial element, where $t \in \Z$ and $v = (v_i)_{i \in \Z}$ with $v_i \in \Z$ for each $i$. If $t$ is non-trivial, then the chosen element survives in the obvious $\Z$ quotient of $G$, which is trivially poly-$\Z$ and virtually Abelian. Assume, then, that $v$ is non-trivial. For each $n \in \Z_{\geqslant 1}$, there is a quotient map
    \[
        \bigoplus_{i \in \Z} \Z \to \bigoplus_{i \in \Z/n} \Z, \quad (u_i)_{i \in \Z} \mapsto \sum_{j \in \Z} \left( u_{jn}, u_{1 + jn}, \dots, u_{n-1 + jn} \right).
    \]
    This map is well-defined, since only finitely many of the elements $u_i$ can be non-trivial. Moreover, this map extends to a quotient map
    \[
        G = \left(\bigoplus_{i \in \Z} \Z\right) \rtimes \Z \to \left(\bigoplus_{i \in \Z/n} \Z\right) \rtimes \Z,
    \]
    where in the quotient, $\Z$ acts by cyclically permuting the coordinates modulo $n$. For each $n$, the quotient is poly-$\Z$, and moreover it is virtually Abelian since the action of $\Z$ has finite order. By choosing $n$ to be sufficiently large, we can ensure that the image of $(v,t)$ is non-trivial in the quotient, which proves that $\Z \wr \Z$ is RFRS.
\end{ex}

%%%
%%% APPLICATIONS
%%%
\section{Applications of the main result} \label{sec: application}

%%% EQUIVALENCES
\subsection{Characterisations of coherent RFRS groups}
In this section we apply \cref{cor: hom coherence implies vanishing} to deduce that many properties are equivalent to coherence within the class of RFRS groups of cohomological dimension two.

We recall that a $2$-complex $X$ has \emph{nonpositive immersions} if for every cellular immersion $Y \looparrowright X$, either $\chi(Y) \leqslant 0$ or $\pi_1(Y) = \{1\}$.

\begin{cor}\label{cor: equivalences}
    Let $G$ be a finitely generated virtually RFRS group and let $k$ be a field such that $\cd_k(G) \leqslant 2$. The following are equivalent:
    \begin{enumerate}
        \item\label{item: coherence} $G$ is coherent;
        \item\label{item: gp alg coh} $k[G]$ is coherent;
        \item\label{item: hom coherence} $G$ is homologically coherent over $k$;
        \item\label{item: H2 fg} $\H_2(H;k)$ is a finite-dimensional $k$-vector space for every finitely generated subgroup $H \leqslant G$;
        \item\label{item: second L2 fg all subgps} $b_2^{\Dk{H}}(H)$ is finite for every finitely generated subgroup $H \leqslant G$;
        \item\label{item: second L2 zero} $b_2^{\Dk{G}}(G) = 0$;
        \item\label{item: v free by Z} $G$ is virtually free-by-cyclic;
        \item\label{item: DkG coherent} $\Dk{G}$ is coherent as a $k[G]$-module;
        \item\label{item: weak dim 1} $\mathcal D_{k[G]}$ is of weak dimension at most one as a $k[G]$-module;
        \item\label{item: NPI} $G$ has a subgroup of finite index isomorphic to $\pi_1(X)$ where $X$ is a $2$-complex with nonpositive immersions.
    \end{enumerate}
\end{cor}
\begin{proof}
    Plainly, \ref{item: coherence} $\Rightarrow$ \ref{item: hom coherence} $\Leftarrow$ \ref{item: gp alg coh} and \ref{item: hom coherence} $\Rightarrow$ \ref{item: H2 fg}. By \cite[Lemma 6.11]{FisherKlinge_RPVN}, we have \ref{item: H2 fg} $\Rightarrow$ \ref{item: second L2 fg all subgps}. By \cite[Corollary 3.5]{JaikinLinton_coherence}, we have \ref{item: second L2 fg all subgps} $\Rightarrow$ \ref{item: hom coherence}, establishing the equivalence of items \ref{item: hom coherence}, \ref{item: H2 fg}, and \ref{item: second L2 fg all subgps}.

    Now, \ref{item: hom coherence} $\Rightarrow$ \ref{item: second L2 zero} by \cref{cor: hom coherence implies vanishing}, and \ref{item: second L2 zero} $\Rightarrow$ \ref{item: v free by Z} by \cite[Theorem E]{Fisher_freebyZ}. In turn, \ref{item: v free by Z} $\Rightarrow$ \ref{item: coherence} by \cite{FeighnHandel_FreeByZCoherent}.
    
    We have \ref{item: v free by Z} $\Rightarrow$ \ref{item: DkG coherent}, \ref{item: weak dim 1}, \ref{item: NPI} by \cite[Lemma 3.4]{Fisher_freebyZ}, \cite[Theorem 6.1]{Wise_JussieuCoherenceNPI}, and \cite[Proposition 3.9]{HennekeLopez_PseudoSyl} respectively. It is clear that \ref{item: weak dim 1} $\Rightarrow$ \ref{item: second L2 zero} and \ref{item: DkG coherent} $\Rightarrow$ \ref{item: gp alg coh}. Finally, \ref{item: NPI} $\Rightarrow$ \ref{item: hom coherence} by \cite[Theorem 1.2]{JaikinLinton_coherence}. \qedhere
\end{proof}

\begin{proof}[Proof of \cref{cor: C}]
    \cref{cor: C} follows from \cref{cor: equivalences} by setting $k = \Q$, and recalling the definition of $\Dk{G}$-Betti numbers and that $\mathcal D_{\Q[G]}$ is isomorphic to the Linnell ring of $\Q[G]$ (see \cref{subsec: L2}). We comment further on items \ref{item: DkG coherent} and \ref{item: weak dim 1}, and how they imply the corresponding statements of \cref{cor: C}. Since $\mathcal D_{\Q[G]}$ is (canonically isomorphic to) a division subring of $\mathcal U(G)$, we have
    \[
        \Tor_i^{\Q[G]}(M,\mathcal U(G)) \cong \Tor_i^{\Q[G]}(M, \mathcal D_{\Q[G]}) \otimes_{\mathcal D_{\Q[G]}} \mathcal U(G)
    \]
    and the left-hand side vanishes if and only if the right-hand side does. Hence, the fact that $\mathcal D_{\Q[G]}$ is of weak dimension at most one as a $\Q[G]$-module if and only if the same property holds for $\mathcal U(G)$. Since $\mathcal U(G)$ is a direct sum of copies of $\mathcal D_{\Q[G]}$ (as a $\Q[G]$-module), it follows that $\mathcal U(G)$ is a coherent $\Q[G]$-module if and only if $\mathcal D_{\Q[G]}$ is. \qedhere
\end{proof}

%%% COXETER GROUPS
\subsection{Coxeter groups} \label{subsec: Coxeter}

We follow the discussion on Coxeter groups from \cite[Section 2]{JankiewiczWise_Coxeter}. For the remainder of this subsection, $G$ will always denote a finitely generated Coxeter group with the presentation $\pres{v_1, \dots, v_n}{(v_iv_j)^{m_{ij}} : 1 \leqslant i \leqslant j \leqslant n}$. Coxeter groups are virtually torsion-free; in fact they are virtually special by \cite{HaglundWise_CoxeterSpecial}, which is a property we will use below. Let $X$ be the presentation complex of the presentation for $G$ given above, and let $\widehat{X} \rightarrow X$ be a finite-degree cover corresponding to a torsion-free subgroup $H \leqslant G$. We define the \emph{compression} $\overline X$ of $\widehat X$ in two steps:
\begin{enumerate}
    \item collapse each $2$-cell with boundary path labelled by $v_i^2$ for some $i = 1, \dots , n$ to a single edge labelled by $v_i$;
    \item identify each $2$-cell with common boundary path $(v_i v_j)^{m_{ij}}$ to a single $2$-cell (there are $2m_{ij}$ such $2$-cells per common boundary path).
\end{enumerate}
We say that $G$ is \emph{two-dimensional} if $\overline{X}$ is aspherical for some finite-index subgroup $H \leqslant G$. This condition is equivalent to the requirement that
\[
    \frac{1}{m_{ij}} + \frac{1}{m_{jk}} + \frac{1}{m_{ki}} \leqslant 1
\]
(which in turn is equivalent to $\langle v_i, v_j, v_k \rangle$ being infinite)
for every triple of pairwise distinct indices $i,j,k$, with the convention that $\frac{1}{\infty} = 0$. It follows that two-dimensional Coxeter groups are of virtual cohomological dimension $2$. A torsion-free finite-index subgroup $H \leqslant G$ is thus of finite type and therefore $G$ has a well-defined Euler characteristic (see \cref{subsec: L2}). The Euler characteristic of a two-dimensional Coxeter group $G$ can be explicitly calculated from the presentation:
\[
    \chi(G) \ = \ 1 - \frac{n}{2} + \sum_{1 \leqslant i < j \leqslant n} \frac{1}{2m_{ij}}.
\]

Using our main result, we confirm Conjecture 3.4 and give a positive solution to Problem 4.6 of \cite{JankiewiczWise_Coxeter} (see also \cite[Problem 11.28]{Wise_anInvitation}). By a \emph{Coxeter subgroup} of $G$, we mean a subgroup $H$ generated by a subset of the generators in the fixed presentation. In this case, $H$ is also a Coxeter group, whose relations are given by those involving the chosen subset of the generators.

\begin{thm}\label{thm: Coxeter}
    Let $G$ be a two-dimensional Coxeter group. Then $G$ is coherent if and only if $\chi(H) \leqslant 0$ for every infinite Coxeter subgroup $H \leqslant G$.
\end{thm}
\begin{proof}
    Since Coxeter groups are virtually special \cite{HaglundWise_CoxeterSpecial}, they are virtually RFRS. If $G$ is coherent, then $\b{2}(H) = 0$ for every subgroup $H \leqslant G$ by \cref{cor: equivalences}. It then follows that $\chi(H) \leqslant 0$ for every infinite Coxeter subgroup $H \leqslant G$. Conversely, if $\chi(H) \leqslant 0$ for every infinite Coxeter subgroup, then $G$ is coherent by \cite[Theorem 1.4, Proposition 6.2]{JaikinLinton_coherence}. \qedhere
\end{proof}

As a corollary, we compute the $L^2$-Betti numbers of coherent two-dimensional Coxeter groups and obtain that they are precisely the virtually free-by-cyclic ones. Note that the computation of $L^2$-Betti numbers of a general Coxeter group is a difficult problem, even in the right-angled case.

\begin{cor}\label{cor: CoxeterL2_NPSC}
    Let $G$ be a coherent infinite two-dimensional Coxeter group. Then $\b{1}(G) = -\chi(G)$ and $\b{n}(G) = 0$ for $n \neq 1$. In particular, $G$ is virtually free-by-cyclic.
\end{cor}

In our next application, we relate coherence of Coxeter groups to the curvature properties of $\overline X$. For this we briefly recall the notions of planar and sectional curvature of a $2$-complex, again following the discussion in \cite[Section 4]{JankiewiczWise_Coxeter}. An \emph{angled $2$-complex} is a combinatorial $2$-complex $Y$ with an assignment of a real number $\angle(e) \in \R$ for each edge of $\link(y)$ for each $0$-cell $y$ of $Y$. We think of these numbers as angles at the corners of the $2$-cells of $X$. The \emph{curvature of a $2$-cell} $f$ is
\[
    \kappa(f) \ := \ 2\pi - \sum (\pi - \angle(e)),
\]
where the sum is taken over the corners $e$ of $f$. The \emph{curvature of a $0$-cell} $y$ is
\[
    \kappa(y) \ := \ 2\big(\pi - \chi(\link(y))\big) + \sum \big(\pi - \angle(e)\big),
\]
where the sum is taken over the corners of $2$-cells meeting $v$.

A \emph{section} of $Y$ at $y$ is a based combinatorial immersion $(S,s) \rightarrow (Y,y)$, where $S$ is another combinatorial $2$-complex and $s$ is a $0$-cell of $S$. The requirement that the map be an immersion is equivalent to asking that the induced maps on links be injective. The section is \emph{planar} if $S$ is homeomorphic to a closed $2$-disk. The \emph{curvature of the section} is $\kappa(s)$, where the angles on $S$ (at $s$) are obtained by pulling back those of $Y$. Then $Y$ has \emph{nonpositive (planar) sectional curvature} if $\kappa(f) \leqslant 0$ for each $2$-cell $f$ of $Y$ and $\kappa(s) \leqslant 0$ for all (planar) sections $(S,s) \rightarrow (Y,y)$.

The complex $\overline{X}$ can be given a natural angled structure by making each $n$-gon Euclidean; more precisely, the corners of the cells corresponding to the relator $(v_iv_j)^{m_{ij}}$ are all given the angle $(1 - \frac{1}{m_{ij}})\pi$. By a direct calculation, it is easy to see that if $\overline{X}$ has nonpositive planar curvature, then $G$ is two-dimensional. The following result confirms a conjecture of Jankiewicz and Wise \cite[Conjecture 4.7]{JankiewiczWise_Coxeter} (see also \cite[Conjecture 11.29]{Wise_anInvitation}).

\begin{thm}\label{thm: conj_Kasia_Dani}
    If $G$ is a Coxeter group with nonpositive planar sectional curvature, then the following are equivalent:
    \begin{enumerate}[label=\textnormal{(\arabic*)}]
        \item $G$ is coherent;
        \item $\overline{X}$ has nonpositive sectional curvature.
    \end{enumerate}
\end{thm}
\begin{proof}
    One direction is done in \cite[Theorem 1.4]{JaikinLinton_coherence}, so assume $G$ is coherent. According to \cite[Theorem 4.5]{JankiewiczWise_Coxeter} it is sufficient to check that for each nontrivial Coxeter subgroup $H \subseteq G$ whose associated graph $\Gamma_H$ is connected but not a tree we have $\chi(H) \leqslant 0$. Since $H$ is coherent, $\b{2}(H) = 0$ by \cref{cor: equivalences}. Since the associated graph of $H$ is not a tree, it must have at least $3$ vertices, so it is infinite and therefore $\b{0}(H) = 0$. Hence, $\chi(H) \leqslant 0$. \qedhere
\end{proof}

%%% RANDOM GROUPS
\subsection{Random groups} \label{subsec: random}

In this subsection we consider the few relator random presentation model of Arzhantseva--Ol'shanski{\u\i}  \cite{ArzhantsevaOlshanskii_fewRelRandom}. Fix integers $m \geqslant 2$, $n \geqslant 0$, and a length $l > 0$. A \emph{random presentation} (with respect to these parameters) is a group presentation 
\[
    \pres{s_1, \dots, s_m}{r_1, \dots, r_n},
\]
where the relators $r_i$ are words sampled uniformly at random from the set of  all cyclically reduced words of length at most $l$ on the alphabet $\{s_1^{\pm1}, \dots, s_m^{\pm1}\}$. By a standard abuse of notation, we will use the presentation to denote the group that it presents. We say that a property $\mathcal P$ of groups holds with asymptotic probability $0 \leqslant \lambda \leqslant 1$ if 
\[
    \mathbb P(\pres{s_1, \dots, s_m}{r_1, \dots, r_n} \ \text{has} \ \mathcal P) \rightarrow \lambda \quad \text{as} \quad l \rightarrow \infty.
\]

We obtain the following result, which essentially confirms \cite[Conjecture 17.15]{Wise_anInvitation}. It only remains to upgrade item \ref{item: zero} below from positive asymptotic probability to asymptotic probability $1$.

\begin{thm}\label{thm: random coherence}
    Let $G = \pres{s_1, \dots, s_m}{r_1, \dots, r_n}$ be a random presentation with relators of length at most $l$.
    \begin{enumerate}[label=\textnormal{(\arabic*)}]
        \item\label{item: positive} $G$ is incoherent with asymptotic probability $1$ when $n > m-1$.
        \item\label{item: zero} $G$ is virtually (finitely generated free)-by-cyclic with positive asymptotic probability when $n = m-1$.
        \item\label{item: negative} $G$ is virtually free-by-cyclic with asymptotic probability $1$ when $n < m-1$.
    \end{enumerate}
\end{thm}
\begin{proof}
    Items \ref{item: zero} and \ref{item: negative} are due to Kielak--Kropholler--Wilkes \cite[Theorems A and B]{KielakKrophollerWilkes_RandomL2}, so we focus on item \ref{item: positive}. By \cite[Lemma 3]{ArzhantsevaOlshanskii_fewRelRandom}, the presentation has the $C'(\frac{1}{6})$ small cancellation property with asymptotic probability $1$, and moreover, the presentation will have no relators that are proper powers with asymptotic probability $1$. In particular, the presentation complex of the random presentation is aspherical \cite{Lyndon_DehnAlg}, $G$ is hyperbolic \cite[Section 4.7]{GromovHG}, $G$ acts properly and cocompactly on a CAT(0) cube complex \cite{Wise_cubulatingSmallCancellation}, and thus $G$ is virtually compact special \cite{AgolHaken} with asymptotic probability $1$. Since the presentation complex is aspherical, we have $\chi(G) = 1 - m + n > 0$ and therefore $\b{2}(G) > 0$. Hence, $G$ is incoherent with asymptotic probability $1$ by \cref{cor: equivalences}. \qedhere
\end{proof}

%%% ALGORITHM
\subsection{An algorithm to detect coherence}

Since $L^2$-Betti numbers are algorithmically computable in many classes of groups with decidable word problem \cite{LohUschold_L2computability}, we are able to deduce the following corollary.

\begin{cor}\label{cor: coh algorithm}
    There is an algorithm that, given a virtually compact special group $G$ with $\cd_\Q(G) \leqslant 2$, determines whether or not $G$ is coherent.
\end{cor}
\begin{proof}
    We first show there is an algorithm that determines the index of a torsion-free subgroup of $G$; we thank Marco Linton for providing this argument. Enumerate all presentations of all finite-index subgroups of $G$, keeping note of the index, and enumerate all presentations of all fundamental groups of compact special cube complexes. Since we are given that $G$ is virtually compact special, eventually two of the presentations will agree. Since special groups are torsion-free, we thus obtain the index of a torsion-free subgroup.

    Since virtually compact special groups are residually finite, they have decidable word problem and satisfy the Determinant Conjecture \cite{ElekSzabo_hyperlinearity}. Hence, by \cite[Theorem 5.7 and Remark 5.8]{LohUschold_L2computability}, there is an algorithm that computes a sequence $(q_n)_{n \in \N}$ of rational numbers such that
    \begin{equation} \label{eq: compute seq}
        \left| \b{2}(G) - q_n \right| < \frac{1}{2^n}
    \end{equation}
    for all $n \in \N$.

    We can now describe the algorithm to detect coherence. Given (a presentation for) a group $G$ as in the statement of the corollary, first determine the index $l$ of a torsion-free subgroup, using the algorithm from the first paragraph. By \cite{Schreve_AtiyahVCS}, we know that $G$ satisfies the Strong Atiyah Conjecture, and hence that $\b{2}(G) \in \frac{1}{l} \Z$. Let $n$ be such that $\frac{1}{2^n} < \frac{1}{2l}$ and use the algorithm from \cite{LohUschold_L2computability} to compute the first $n$ terms of the sequence $(q_n)$ satisfying \eqref{eq: compute seq}. If $q_n \leqslant \frac{1}{2l}$, then 
    \[
        \b{2}(G) < \frac{1}{2^n} + q_n < \frac{1}{2l} + \frac{1}{2l} = \frac{1}{l},
    \]
    which forces $\b{2}(G) = 0$. The algorithm then terminates and outputs ``$G$ is coherent". On the other hand, if $q_n > \frac{1}{2l}$, then
    \[
        \b{2}(G) > q_n - \frac{1}{2^n} > 0,
    \]
    and so the algorithm terminates and outputs ``$G$ is incoherent". The correctness of this algorithm is guaranteed by \cref{cor: equivalences}. \qedhere
\end{proof}

%%% RIGIDITY
\subsection{Rigidity of coherence}

The question as to whether coherence is a quasi-isometry invariant is a well known open problem (it is asked, e.g., in \cite[Problem 35]{Wise_anInvitation}). Here, we show that this is the case within the class of virtually RFRS groups of cohomological dimension at most $2$. Since the vanishing of $L^2$-Betti numbers has good invariance properties, we are also able to deduce results on measure equivalence and profinite invariance of coherence.

\begin{cor}\label{cor: rigidity}
    Let $G$ and $H$ be finitely generated virtually RFRS groups of virtual cohomological dimension at most $2$. If $G$ and $H$ are
    \begin{enumerate}
        \item\label{item: QI} quasi-isometric,
        \item\label{item: ME} measure equivalent, or
        \item\label{item: profinite equivalence} virtually compact special and profinitely isomorphic,
    \end{enumerate}
    then $G$ is coherent if and only if $H$ is coherent.
\end{cor}

We will denote by $\widehat G$ the profinite completion of $G$, i.e.\ the topological group 
\[
    \varprojlim_{H \trianglelefteqslant_{\mathrm{f.i.}} G} G/H,
\]
where $H$ runs over all finite-index normal subgroups of $G$. A group is called \emph{good} in the sense of Serre if the natural map 
\[
    \H_{\mathrm{cts}}^i \left( \widehat G ; M \right) \to \H^i(G; M)
\]
is an isomorphism for all $i$ and all finite $G$-modules $M$. The cohomology $\H_{\mathrm{cts}}^i$ denotes the continuous cohomology of $\widehat G$, which means that cohomology is computed with respect to continuous cochains, where $M$ carries the discrete topology.

\begin{rem}
    For the proofs of items \ref{item: QI} and \ref{item: ME}, the virtual cohomological dimension at most $2$ assumption can be weakened to rational cohomological dimension at most $2$. Moreover, \ref{item: profinite equivalence} holds with the assumption ``virtually compact special" replaced with ``virtually \{RFRS of type $\FP(\Z)$\} and good in the sense of Serre". For instance, Coxeter groups are known to be good by \cite[Proposition 3.6]{HughesMollerVarghese_CoxterProfiniteProps}, and as mentioned above they are virtually RFRS, so \cref{cor: rigidity}\ref{item: profinite equivalence} applies to two-dimensional Coxeter groups. It is not known whether Coxeter groups are virtually \emph{compact} special.
\end{rem}

\begin{proof}[Proof of \cref{cor: rigidity}]
    If $G$ and $H$ are quasi-isometric (resp.\ measure equivalent), then $\b{2}(G) = 0$ if and only if $\b{2}(H) = 0$ by \cite[Corollary 6.3]{MOSS_qil2vanishing} (resp.\ \cite[Theorem 6.3]{Gaboriau2002}). By \cref{cor: equivalences}, $G$ is coherent if and only if $H$ is.

    Now suppose that $G$ and $H$ are both virtually compact special and have isomorphic profinite completions. We may assume that both $G$ and $H$ are infinite so that their zeroth $L^2$-Betti numbers vanish, and we may also assume that $\cd(G) = \cd(H) \leqslant 2$ by passing to suitable torsion-free finite-index subgroups with isomorphic profinite completions (see, e.g., \cite[Theorem 1.3.20]{Wilkes_ProfiniteBook}). Since virtually compact special groups are good \cite[Corollary 6.3]{Schreve_AtiyahVCS}, we have
    \begin{align*}
        \chi(G) &= b_0(G;\mathbb F_p) - b_1(G;\mathbb F_p) + b_2(G;\mathbb F_p) \\
        &= b_0(\widehat G;\mathbb F_p) - b_1(\widehat G;\mathbb F_p) + b_2(\widehat G;\mathbb F_p) \\
        &= b_0(\widehat H;\mathbb F_p) - b_1(\widehat H;\mathbb F_p) + b_2(\widehat H;\mathbb F_p) \\
        &= b_0(H;\mathbb F_p) - b_1(H;\mathbb F_p) + b_2(H;\mathbb F_p) = \chi(H).
    \end{align*}
    for all primes $p$. Since the first $L^2$-Betti number is a profinite invariant \cite[Corollary 3.3]{BridsonReidConder_FuchsianProfinite}, we have
    \[
        \b{2}(G) = \chi(G) + \b{1}(G) = \chi(H) + \b{1}(H) = \b{2}(H).
    \]
    Therefore, \cref{cor: equivalences} implies that $G$ is coherent if and only if $H$ is coherent. \qedhere
\end{proof}

%%% PARAFREE
\subsection{The Parafree Conjecture and coherence}

A group $G$ is \emph{parafree} if it is residually nilpotent and there exists a free group $F$ such that $G/\gamma_n(G) \cong F/\gamma_n(F)$ for all $n \geqslant 1$. The groups $\gamma_n(G)$ are the lower central series terms of $G$ and are defined recursively by $\gamma_1(G) = G$ and $\gamma_{n+1}(G) = [G,\gamma_n(G)]$ for all $n \geqslant 1$. By a result of Reid \cite[Theorem 9.2]{Reid_parafreeRFRS}, parafree groups are RFRS.

The most famous open problem in the study of parafree groups is Baumslag's Parafree Conjecture, which predicts that $\H_2(G;\Z) = 0$ for all finitely generated parafree groups $G$. The Strong Parafree Conjecture predicts that $\H_2(G;\Z) = 0$ and $\cd(G) \leqslant 2$ for every finitely generated parafree group $G$. Both versions of the conjecture are open, and there are many examples of parafree groups of cohomological dimension at most two with vanishing second integral homology.

As a consequence of our main result, we show that in cohomological dimension two, the Parafree Conjecture is equivalent to the lesser known conjecture that parafree groups are coherent.

\begin{cor}\label{cor: parafree}
    If $G$ is a parafree group with $\cd(G) \leqslant 2$, then $\H_2(G;\Z) = 0$ if and only if $G$ is coherent.
\end{cor}
\begin{proof}
    By \cite[Corollary D]{FisherKlinge_RPVN}, if $\H_2(G;\Z) = 0$, then $G$ is coherent. Conversely, if $G$ is coherent, then $b_2^{\Dk{G}}(G) = 0$ for all fields $k$ by \cref{cor: hom coherence implies vanishing}. By the proof of \cite[Corollary D]{FisherKlinge_RPVN}, this implies that $\H_2(G;\Z) = 0$. \qedhere
\end{proof}

\appendix
\section{Cubulated locally quasi-convex hyperbolic groups are virtually free-by-cyclic}\label{sec: appendix}
\smallskip
\begin{center}by \textsc{Marco Linton}\end{center}
\medskip

In this appendix we use \cref{thm: B} to prove the following theorem.

\begin{thm}
\label{main}
If $G$ is a virtually compact special and locally hyperbolic group, then $G$ is virtually an extension of a free product of free and surface groups by $\Z$.
\end{thm}

We point out that extensions of free products of free and surface groups by $\Z$ have cohomological dimension at most 3. Thus \cref{main} yields a bound on the virtual cohomological dimension of virtually compact special groups in which every finitely generated subgroup is hyperbolic. This should be compared with a question of Wise \cite[Problem 28]{Wise_anInvitation} which asks whether coherent torsion-free hyperbolic groups of cohomological dimension at least four exist.

\begin{thm}
\label{main_cor}
If $G$ is a virtually compact special locally hyperbolic group, then $G$ is locally quasi-convex if and only if $G$ contains no \{finitely generated free\}-by-$\Z$ or \{closed surface\}-by-$\Z$ subgroups. 

In particular, if $G$ is a virtually compact special and locally quasi-convex hyperbolic group, then $G$ is virtually free-by-cyclic.
\end{thm}

Note that by work of Wise \cite{Wise_structureQCH}, a locally quasi-convex hyperbolic group which admits a hierarchy is virtually compact special. In particular, such groups must be virtually free-by-cyclic by \cref{main_cor} and thus have virtual cohomological dimension at most 2. This solves a problem of Wise \cite[Problem 29]{Wise_anInvitation} in the cubulated case. \cref{main_cor} also solves \cite[Conjecture 1.6]{AW26} in the cubulated case.

It is tempting to want to relax the assumption that $G$ be locally hyperbolic in \cref{main} to the assumption that $G$ be coherent. We show that this can be done in \cref{prop:generalise} if there exists no compact special hyperbolic group $N\rtimes \Z$ so that $\cd_{\Q}(N)\in \{2, 3\}$ and $N$ is of type $F$ but not hyperbolic. Recall that a group $G$ is said to have type $F$ if it admits a compact classifying space. We point out that Italiano--Martelli--Migliorini showed in \cite{IMM_5mfld} that there is such a group, but with $\cd(N) = 4$. 

\subsection{The proof}

We first state all the ingredients we shall need for the proof of \cref{main}. We begin by restating \cref{thm: B} with the hypotheses we shall use.

\begin{thm}
\label{vanishing}
If $G$ is a virtually special group in which every finitely generated subgroup has type $\FP_{\infty}(\Q)$, then $b_i^{(2)}(G) = 0$ for all $i>1$.
\end{thm}

The second ingredient we shall need is due to Kielak and the author \cite[Proposition 6.4]{KielakLinton_FbyZ}.

\begin{prop}
\label{embedding}
If $G$ is a non-free hyperbolic and compact special group with $b_i^{(2)}(G) = 0$ for all $i\geqslant 2$, then there is an acylindrical HNN-extension $L = G*_{\psi}$ of $G$ over a quasi-convex free subgroup so that $L$ is hyperbolic and virtually compact special and $b_i^{(2)}(L) = 0$ for all $i$.
\end{prop}

Our third ingredient is Fisher's generalisation \cite[Theorem A]{Fisher_Improved} of Kielak's virtual fibring theorem \cite{KielakRFRS}.

\begin{thm}
\label{fibring}
If $G$ is a compact special group with $b_i^{(2)}(G) = 0$ for $i\leqslant n$, then $G$ has a finite index subgroup $H$ so that $H\cong N\rtimes \Z$ with $N$ of type $\FP_n(\Q)$. 
\end{thm}

We shall also need a criterion for when a subgroup of a hyperbolic group of type $\FP_2(\Q)$ is hyperbolic. This will allow us to conclude that the subgroup $N$ from \cref{fibring} is hyperbolic under certain hypotheses.

\begin{lem}
\label{local_hyp}
Let $G$ be a hyperbolic group and suppose that $G$ acts acylindrically on a tree $T$ so that vertex stabilisers are finitely generated locally hyperbolic and edge stabilisers are quasi-convex in $G$ and locally quasi-convex. Then subgroups of $G$ of type $\FP_2(\Q)$ are hyperbolic.
\end{lem}

\begin{proof}
Let $H\leqslant G$ be a subgroup of type $\FP_2(\Q)$. By \cite[Theorem 4.4]{JaikinLinton_coherence}, since $H$ has type $\FP_2(\Q)$, $H$ acts cocompactly on a tree $S$ with finitely generated vertex and edge stabilisers dominating $T$. Since vertex stabilisers for $T$ are locally hyperbolic, the vertex stabilisers for $S$ are hyperbolic. Since finitely generated subgroups of edge stabilisers of $T$ are quasi-convex in $G$, edge stabilisers of $S$ are quasi-convex in $G$ and therefore they are also quasi-convex in the vertex stabilisers of $S$. Since the edge stabilisers of $S$ are quasi-convex in $G$, they have finite height in $G$ and hence also in $H$ (see \cite{GMRS98}). This implies that the action of $H$ on $S$ is acylindrical. Now by the Bestvina--Feighn combination theorem \cite{BF92} we see that $H$ is hyperbolic.
\end{proof}

We now have all the necessary ingredients to prove our main theorem.

\begin{proof}[Proof of \cref{main}]
By \cref{vanishing} we see that $b_i^{(2)}(G) = 0$ for all $i>1$. By \cref{embedding}, there is a finite index subgroup $H\leqslant G$ and a hyperbolic and virtually compact special group $L$ with $b_i^{(2)}(L) = 0$ for all $i$ so that $L \cong H*_{\psi}$ and where the Bass--Serre tree $T$ for the action is acylindrical and all edge stabilisers are free. By \cref{fibring}, $L$ contains a finite index subgroup of the from $N\rtimes \Z$ with $N$ of type $\FP_2(\Q)$. By \cref{local_hyp}, we see that $N$ is hyperbolic. By a result of Dahmani--Krishna--Mutanguha \cite[Proposition 5.4]{DKM25}, $N$ contains a finite index characteristic subgroup $N'$ which is a free product of free and surface groups. Thus, $L$ contains a finite index subgroup $L'\cong N'\rtimes \Z$. Thus $G\cap L'$ is a finite index subgroup of $G$ of the required form.
\end{proof}

\begin{proof}[Proof of \cref{main_cor}]
If $G$ is locally quasi-convex, then $G$ cannot contain any subgroup of the form $N\rtimes \Z$ with $N$ non-trivial finitely generated as $N$ would be a non quasi-convex subgroup of $G$ (see \cite[Proposition 3.16]{BridsonHaefliger_thebook}). Now suppose that $G$ contains no $F_n\rtimes \Z$ or $\pi_1(S)\rtimes\Z$ subgroups. \cref{main} implies that $G$ contains a finite index subgroup $H$ so that $H\cong N\rtimes \Z$ where $N\cong F*(\ast_{i\in I}\pi_1(S_{i}))$ where $F$ is free and each $S_{i}$ is a closed surface. Let $\psi\colon N\to N$ be the automorphism defining the semidirect product. If for some $i\in I$ and some $n\geqslant 1$, the subgroup $\psi^n(\pi_1(S_{i}))$ is conjugate to $\pi_1(S_{i})$, then $G$ contains a \{closed surface\}-by-cyclic subgroup, a contradiction. Thus, for each $i\in I$ and each $n\geqslant 1$ we have that $\psi^n(\pi_1(S_{i}))$ is not conjugate to $\pi_1(S_{i})$ within $N$. Consider the $H$-set $X = \{\pi_1(S_i)^h \mid i\in I, h\in H\}$. By \cite[Proposition 6.5]{JaikinLintonSanchez_OneRelProd} we have that $\cd_{\Q}(H, X)\leqslant 2$. By the same result, we also have that for each $x\in X$, $\Stab(x)\cong \pi_1(S_i)$ for some $i\in I$. Since $\cd_{\Q}(\pi_1(S_i)) = 2$ for each $i\in I$ and $\cd_{\Q}(H, X)\leqslant 2$, by observing the long exact sequence arising from the short exact sequence $0\to\omega_{\Q}(X)\to \Q[X]\to \Q\to 0$, we see that $\cd_{\Q}(H)\leqslant 2$ (see \cite[Lemma 3]{Al91} for instance). Combined with the fact that $b_2^{(2)}(H) = 0$ by \cref{vanishing}, we obtain that $H$, and thus $G$, is virtually free-by-cyclic by \cite{KielakLinton_FbyZ}. Finally, the local quasi-convexity of $G$ follows from \cite{Li25}.

For the final statement, we combine the above with the fact that a locally quasi-convex hyperbolic group is locally hyperbolic, see \cite[Proposition 3.9]{BridsonHaefliger_thebook} for instance.
\end{proof}

\subsection{An attempt at strengthening Theorem \ref{main}}

In this section we make explicit the obstruction to generalising our arguments from \cref{main}.

\begin{lem}
\label{lem:type_F}
Let $G$ be a group acting on a tree $T$ in which finitely generated subgroups of vertex stabilisers have type $F$. Every finitely generated subgroup of $G$ of type $\FP_2(\Q)$ has type $F$.
\end{lem}

\begin{proof}
Let $H\leqslant G$ be a subgroup of type $\FP_2(\Q)$. By \cite[Theorem 4.4]{JaikinLinton_coherence}, there is a cocompact $H$-tree $S$ dominating $T$ in which vertex and edge stabilisers of $S$ are finitely generated (and contained in vertex and edge stabilisers of $T$). By assumption, the vertex and edge stabilisers of $S$ have type $F$. Since a finite graph of compact aspherical spaces is compact aspherical, we may construct a compact $K(H, 1)$ for $H$ out of the compact $K(H_v, 1)$ spaces for the vertex groups and $K(H_e, 1)$ spaces for the edge groups of the finite quotient graph of groups $H\backslash S$. Thus, $H$ has type $F$ as claimed.
\end{proof}

\begin{rem}
\label{rem:v_special}
If $G$ is a compact special group, then $G$ admits a quasi-convex hierarchy in the sense of Wise \cite{Wise_structureQCH}. Thus, by inductively applying \cref{lem:type_F} we see that any compact special group that is homologically coherent has all its finitely generated subgroups of type $F$.
\end{rem}

\begin{prop}
\label{prop:generalise}
One of the following holds:
\begin{enumerate}
\item\label{itm:1} Coherent virtually compact special hyperbolic groups are virtually extensions of free products of free and surface groups by $\Z$.
\item\label{itm:2} There exists a compact special hyperbolic group $N\rtimes \Z$ with $\cd_{\Q}(N) \in\{2, 3\}$ so that $N$ has type $F$, but is not hyperbolic.
\end{enumerate}
\end{prop}

\begin{proof}
Assuming no group as in \ref{itm:2} exists we shall prove \ref{itm:1}. So let $G$ be a virtually compact special hyperbolic group in which every finitely generated subgroup has type $\FP_{2}(\Q)$. After passing to a finite index subgroup, we may assume that $G$ is compact special. Since $G$ is compact special, \cref{rem:v_special} implies that every finitely generated subgroup of $G$ has type $F$. In particular, by \cref{vanishing} we have $b_i^{(2)}(H) = 0$ for all $i>1$ and every finitely generated subgroup $H\leqslant G$ (since subgroups of special groups are special).

Suppose that $\cd_{\Q}(G)= 3$. By \cref{embedding}, there is a finite index subgroup $G'\leqslant G$ and a hyperbolic and virtually compact special group $L\cong G'*_{\psi}$ with $b_i^{(2)}(L) = 0$ for all $i$ and with $\cd_{\Q}(L) = \cd_{\Q}(G) = 3$. As in the proof of \cref{main}, $L$ contains a finite index subgroup $L'\cong N\rtimes \Z$ with $N$ of type $\FP_{\infty}(\Q)$ and $\cd_{\Q}(N) = 2$. Since $N$ has type $\FP_2(\Q)$ we may apply \cref{lem:type_F} to the splitting of $N$ induced by the HNN-extension $L\cong G'*_{\psi}$ and obtain that $N$ has type $F$. But this implies that $N$ is hyperbolic by assumption. Now we conclude as in the proof of \cref{main} (using \cite[Proposition 5.4]{DKM25}) that $L'$ and thus $G$ is virtually an extension of a free product of free and surface groups by $\Z$. 

Now suppose that $\cd_{\Q}(G)\geqslant 4$. Then $G$ contains a quasi-convex subgroup $H\leqslant G$ (a vertex group appearing in the quasi-convex hierarchy) with $\cd_{\Q}(H) = 4$. Since $H$ is quasi-convex, it is also hyperbolic and virtually compact special. By applying the same argument as above to $H$, we obtain a finite index subgroup $H'$ which embeds in a group $L' = N\rtimes \Z$ with $\cd_{\Q}(N) = \cd_{\Q}(L) - 1 = 3$ and so that $N$ has type $F$. By our assumption, $N$ is hyperbolic. But this contradicts \cite[Proposition 5.4]{DKM25}.
\end{proof}

\bibliography{bib}
\bibliographystyle{alpha}

\end{document}